\documentclass[10pt,reqno]{amsart}

\usepackage[a4paper,left=35mm,right=35mm,top=30mm,bottom=30mm,marginpar=25mm]{geometry}

\usepackage{amsmath}
\usepackage{amssymb}
\usepackage{amsthm}
\usepackage[utf8]{inputenc}
\usepackage{eurosym}
\usepackage{graphicx}
\usepackage{hyperref}
\usepackage{dsfont}
\usepackage{dsfont}
\usepackage{xcolor,colortbl}
\usepackage{comment}

\allowdisplaybreaks

\usepackage{hyperref}

\usepackage{ifthen}

\makeindex

\newcommand{\supp}{\operatorname{supp}}

\newcommand{\Rr}{{\mathbb{R}}}

\newcommand{\Nn}{{\mathbb{N}}}

\newcommand{\Ee}{{\mathds{E}}}

\newcommand{\Pp}{{\mathcal{P}}}

\def\dx{{\rm d}x}
\def\dy{{\rm d}y}

\def\dt{{\rm d}t}
\def\leq{\leqslant}
\def\geq{\geqslant}

\numberwithin{equation}{section}

\newtheoremstyle{thmlemcorr}{10pt}{10pt}{\itshape}{}{\bfseries}{.}{10pt}{{\thmname{#1}\thmnumber{
#2}\thmnote{ (#3)}}}
\newtheoremstyle{thmlemcorr*}{10pt}{10pt}{\itshape}{}{\bfseries}{.}\newline{{\thmname{#1}\thmnumber{
#2}\thmnote{ (#3)}}}
\newtheoremstyle{defi}{10pt}{10pt}{\itshape}{}{\bfseries}{.}{10pt}{{\thmname{#1}\thmnumber{
#2}\thmnote{ (#3)}}}
\newtheoremstyle{remexample}{10pt}{10pt}{}{}{\bfseries}{.}{10pt}{{\thmname{#1}\thmnumber{
#2}\thmnote{ (#3)}}}
\newtheoremstyle{ass}{10pt}{10pt}{}{}{\bfseries}{.}{10pt}{{\thmname{#1}\thmnumber{
A#2}\thmnote{ (#3)}}}

\theoremstyle{thmlemcorr}
\newtheorem{theorem}{Theorem}
\numberwithin{theorem}{section}
\newtheorem{lemma}[theorem]{Lemma}

\newtheorem{proposition}[theorem]{Proposition}

\theoremstyle{thmlemcorr*}
\newtheorem{theorem*}{Theorem}
\newtheorem{lemma*}[theorem]{Lemma}
\newtheorem{corollary*}[theorem]{Corollary}
\newtheorem{proposition*}[theorem]{Proposition}
\newtheorem{problem*}[theorem]{Problem}
\newtheorem{conjecture*}[theorem]{Conjecture}

\theoremstyle{defi}

\newtheorem{hyp}{Assumption}

\theoremstyle{remexample}
\newtheorem{remark}[theorem]{Remark}

\theoremstyle{ass}

\begin{document}

\title{A duality approach to a price formation MFG model}

\author{Yuri Ashrafyan}
\author{Tigran Bakaryan}
\author{Diogo Gomes}
\author{Julian Gutierrez}
\thanks{King Abdullah University of Science and Technology (KAUST), CEMSE Division, Thuwal 23955-6900. Saudi Arabia. e-mail: julian.gutierrezpineda@kaust.edu.sa}%




\thanks{
      The authors were partially supported by KAUST baseline funds and 
 KAUST OSR-CRG2017-3452.
}
\date{\today}

\begin{abstract}
We study the connection between the Aubry-Mather theory and a mean-field game (MFG) price-formation model. We introduce a framework for Mather measures that is suited for constrained time-dependent problems in $\Rr$. Then, we propose a variational problem on a space of measures, from which we obtain a duality relation involving the  MFG problem examined in \cite{gomes2018mean}. 
\end{abstract}

\maketitle

\section{introduction}
This paper studies the connection between Aubry-Mather theory and certain mean-field games (MFG) that model price formation.
 More precisely, we consider the MFG system
\begin{equation}\label{eq:MFG system}
	\begin{cases}
	-u_t(t,x) + H(x, \varpi(t) + u_x(t,x) )=0 &\quad (t,x)\in [0,T]\times \Rr
	\\
	m_t(t,x) - \big(H_p(x, \varpi(t) + u_x(t,x))m(t,x)\big)_x =0 &\quad (t,x)\in [0,T]\times \Rr
	\\
	-\int_{\Rr} H_p(x, \varpi(t) + u_x(t,x))m(t,x) dx = Q(t)&\quad t\in [0,T]
	\end{cases} ,
\end{equation}
subject to initial-terminal conditions
\begin{equation}\label{eq:initial terminal}
	\begin{cases}
	u(T,x)=u_T(x)
	\\
	m(0,x)=m_0(x)
	\end{cases}\quad x\in \Rr,
\end{equation}
where $Q$, $u_T$, and $m_0$ are given functions, $m_0$ is a probability measure on $\Rr$, and the triplet $(u,m,\varpi)$ is the unknown. Here, the state of a typical agent
is the variable $x\in \Rr$ and represents the assets of that agent. The distribution of assets in 
the population of the agents at time $t$ is encoded in the probability measure $m(\cdot, t)$.
The agents change their assets by trading at a price $\varpi(t)$. The trading is subject to a balance condition encoded in the
third equation in \eqref{eq:MFG system}.  This integral constraint that guarantees supply
$Q(t)$ meets demand is represented by the term on the left-hand side of that condition.  

As introduced in \cite{gomes2018mean}, $u$ is the value function of an agent who trades a commodity with supply $Q$ and price $
\varpi$. The function $u$ is characterized by the first equation in \eqref{eq:MFG system} and the terminal condition in \eqref{eq:initial terminal}. Each agent selects their trading rate in order to minimize a given cost functional (see \eqref{eq: Min problem DS} below). The optimal control selection is $-H_p(x,\varpi(t)+u_x(t,x))$. Under this optimal control, the density $m$ describing the population of agents evolves according to the second equation in \eqref{eq:MFG system} and the initial condition in \eqref{eq:initial terminal}. The third equation in \eqref{eq:MFG system}, which we refer to as the balance condition, is an integral constraint that guarantees supply meets demand. 

\begin{remark}\label{Remark: Notion Sol MFG}
The notion of solutions of \eqref{eq:MFG system} and \eqref{eq:initial terminal} we consider is the following: $u\in C([0,T]\times \Rr)$ solves the first equation in the viscosity sense, $m\in C([0,T],\Pp(\Rr))$ solves the second equation in the distributional sense, and $\varpi\in C([0,T])$.

The system \eqref{eq:MFG system} and \eqref{eq:initial terminal} corresponds to the case $\epsilon=0$ studied in \cite{gomes2018mean}. Under Assumptions \ref{hyp:Q C infty}, \ref{hyp: H separability DS} and \ref{hyp: V-uT Lipschitz 2nd D bounded DS} (see Section \ref{Sec: Assupmtions}), the authors used a fixed-point argument and the vanishing viscosity method to prove the existence of a solution $(u,m,\varpi)$, where $u$ is Lipschitz and semiconcave in $x$, and differentiable $m$-almost everywhere, $m\in C([0,T],\Pp(\Rr))$ w.r.t. the $1$-Wasserstein distance, and $\varpi \in W^{1,1}([0,T])$. Furthermore, under Assumption \ref{hyp: V-uT convex DS}, they obtained uniqueness of solutions, further differentiability of $u$ in $x$ for every $x$, and the boundedness of $u_{xx}$ and $m$. 
\end{remark}

The connection between 
Hamilton-Jacobi equations and Aubtry-Mather theory 
is now well established; see, for example, \cite{Mane2}, \cite{FATH1, FATH2, FATH3, FATH4},
\cite{EGom1, EGom2,BGom}, or \cite{MTL17}. In particular, several generalizations
of Aubry-Mather theory were developed to address problems like diffusions and
study second-order Hamilton-Jacobi equations 
\cite{G, G10}. In particular, 
duality methods, since the pioneering papers 
in \cite{LewisVinter} and \cite{FlemingVermes2} have been explored in multiple
contexts, see for example \cite{MR2883292}. 
Of great interest are the applications to the selection problem 
in the vanishing discount case, 
\cite{MR2458239}, 
\cite{MR3581314}
\cite{2016arXiv160507532G}
and
\cite{MR4175148} and to the large time behavior of Hamilton-Jacobi equations \cite{CGMT}, \cite{gomes2020large}.  
Recently applications of Aubry-Mather theory were developed for MFGs in 
\cite{Cd1} to study long-time behavior, and in 
 \cite{cannarsa2020weak}, where the authors construct Mather measures to prove the existence of solutions for ergodic first-order MFG systems with state constraints.

The prototype MFG system corresponds to an optimal control problem for an agent who optimizes a cost function that depends on the aggregate behavior of other agents encoded in the population distribution $m$. In \cite{gomes2018mean}, the optimal control setting of the MFG system \eqref{eq:MFG system} and \eqref{eq:initial terminal} corresponds to an agent interacting with the population through the price. At the same time, the balance condition between demand and supply is satisfied. This type of interaction arises in price formation models, where the commodity price being traded is an endogenous rather than an exogenous variable. 

Price formation models were studied previously in  
 \cite{BS02} and \cite{BS10}  in the contex of revenue maximization by a producer. 
 Earlier price models  in the context of mean-field games include \cite{MR2835888,MR2817378,MR2573148, MR3210751} and \cite{lachapelle2016efficiency}. 
Applications to electricity markets were examined in 
\cite{matoussi2018extended, ATM19, FTT20} 
and \cite{TBD20}. 
Price models with a market clearing condition were introduced in 
 \cite{gomes2018mean}, \cite{SummerCamp2019}, \cite{GoGuRi2021},
\cite{JSF20} and \cite{FT20}. 
The former work addresses a model for solar renewable energy certificate markets. Finally, 
 \cite{fujii2021equilibrium} examines the effect of a major player.

%
%
%
%

The variational problem that we consider is a relaxed version of the Lagrangian formulation introduced in \cite{gomes2018mean} to derive \eqref{eq:MFG system} and \eqref{eq:initial terminal}. We prove a duality formula (Theorem \ref{thm: Representation}) between solutions of the MFG system and minimizers of a variational problem in the set of generalized Mather measures. For that, we begin by introducing the Legendre transform, $L$,  of $H$; that is,
\begin{equation}\label{eq:Legendre transform}
	L(x,v)=\sup_{p\in\Rr} \left\{-pv - H(x,p)\right\}.
\end{equation}
Our variational problem is
\begin{equation}\label{eq:Min problem m0}
	\inf_{\mu \in\mathcal{H}(m_0)}  \int_0^T \int_{\Rr^2} L(x,v) + v u'_T(x)  ~d\mu(t,x,v),
\end{equation}
where $\mathcal{H}(m_0)$ is the set of admissible measures. These measures are Radon positive measures on $[0,T]\times\Rr^2$ that satisfy the following three conditions. First, the moment condition 
\[
	\int_{[0,T]\times\Rr^2} (|x|^{\zeta_1}+|v|^{\zeta_2} )d\mu(t,x,v)< \infty,
\]
where $\zeta_1$ and $\zeta_2$ depend on the growth of the Hamiltonian 
in Assumption \ref{hyp: H growth} and 
 satisfy condition \ref{eq: zetas selection}. Second, for some probability measure $\nu$ on $\Rr$, the Radon measure verifies
\begin{equation}
\label{holc}
	\int_{[0,T]\times\Rr^2} \varphi_t (t, x) + v \varphi_x(t, x) d\mu (t,x,v)  =\int_{\Rr} \varphi (T, x) d\nu - \int_{\Rr} \varphi (0, x) dm_0
\end{equation}
for all suitable test functions $\varphi$. We refer to the previous as the holonomy condition, as it is motivated by the holonomy condition introduced in \cite{Mane2}. Lastly, the admissible measures satisfy the following balance condition 
\[
	\int_{[0,T]\times\Rr^2} \eta(t) (v-Q(t)) d\mu (t,x,v) =0
\]
for all $\eta$ continuous. If $u_T\in C^1(\Rr)$ with $u'_T$ bounded (see Assumption \ref{hyp: uT C1 Linfty}), the holonomy condition applied to $\varphi(t,x)=u_T(x)$ (see \eqref{eq:holonomy equality}) provides the identity
\[
	\int_{[0,T]\times\Rr^2} v u'_T(x) d\mu (t,x,v)  =\int_{\Rr} u_T(x) d\nu - \int_{\Rr} u_T(x) dm_0.
\]
Using the previous identity, the variational problem \eqref{eq:Min problem m0}  is equivalent to
\begin{equation}\label{eq:Min problem}
	\inf_{\substack{\nu \in \Pp (\Rr) \\
	\mu \in \mathcal{H}(m_0,\nu)}} \int_0^T \int_{\Rr^2} L(x,v) ~d\mu(t,x,v) + \int_{\Rr} u_T(x) d\nu(x),
\end{equation}
where $\mathcal{H}(m_0,\nu)$ is the set of measures that satisfy the moment condition, the holonomy condition for some probability measure $\nu$ on $\Rr$, and the balance condition. 
The difference between \eqref{eq:Min problem} and \eqref{eq:Min problem m0} is the term $-\int_{\Rr} u_T(x) dm_0$, which is independent of $\mu$.

The motivation for this relaxed problem is as follows.  
In \cite{gomes2018mean}, each agent selects a control variable $\alpha$ aiming to solve
\begin{equation}\label{eq: Min problem DS}
	\inf_{\alpha \in \mathcal{A}} \int_0^T L(x(t),\alpha(t)) +\varpi(t) \alpha(t)dt + u_T(x(T)),
\end{equation}
where $\dot{x}(t)=\alpha(t)$, and $\mathcal{A}$, the set of bounded measurable
functions, is the set of admissible controls. The price $\varpi$ is chosen so that the aggregate supply meets the demand. Here, 
following Mather's theory (see 
for example \cite{MR2458239}), 
we introduced a relaxed version of problem \eqref{eq: Min problem DS}. This relaxation is problem
\eqref{eq:Min problem}. The key idea is that each optimal trajectory $t\mapsto x^*(t)$ with optimal control $t\mapsto \alpha^*(t)$ solving \eqref{eq: Min problem DS} defines the measure $d\mu^*(t,x,v)=dt \times d_{\delta_{x^*(t)}}\times d_{\delta_{\alpha^*(t)}}$. This measure is supported on the path $t\mapsto (x^*(t),\alpha^*(t))$
and satisfies \eqref{holc}. Accordingly, the function in \eqref{eq: Min problem DS}
becomes
\begin{align*}
	\int_0^T \int_{\Rr^2} L(x,v) +\varpi v ~d\mu^*(t,x,v) + \int_{\Rr} u_T(x) d\nu^*(x)
\end{align*}
where $d\nu^*=\delta_{x^*(T)}$; that is, the variational cost for the measure equals the variational cost for the optimal trajectory. 
%

Our first result for the variational problem on measures \eqref{eq:Min problem} is a duality formula between minimizing measures and Hamilton-Jacobi equations that involves the following function.
Let $h:\mathcal{P}(\Rr) \times \mathcal{P}(\Rr) \to \Rr \cup \{+\infty\}$ be
\begin{equation}\label{def: h}
	h(m_0, \nu) =
	\begin{cases} 
		\displaystyle 	\inf_{\substack{\mu \in \mathcal{H}(m_0,\nu)}}  \int_{\Omega} L(x,v) + v u_T'(x)d\mu(t,x,v),
		& \text{if }  \mathcal{H}(m_0,\nu) \neq \emptyset,
		\\ 
		+\infty & \text{if }  \mathcal{H}(m_0,\nu) = \emptyset.
	\end{cases}
\end{equation}
The main assumptions on $L$ and $u_T$ are stated in Section \ref{Sec: Assupmtions}, 
after which, in Section \ref{Sec:Duality}, we develop a framework of Mather measures  suitable for the MFG system \eqref{eq:MFG system} and \eqref{eq:initial terminal}. Finally, 
in that section,  we prove the following  theorem. 

\begin{theorem}\label{thm:convex duality h}
Let $h$ be given by \eqref{def: h} and let $\zeta$ satisfy \eqref{eq: L growth out of H}. Suppose Assumptions \ref{hyp: H convex}-
\ref{hyp:Q C infty} hold. Assume that $\nu_T \in \mathcal{P}(\Rr)$ is such that $\mathcal{H}(m_0,\nu_T) \neq \emptyset$.
Then,
\begin{align*}
	h(m_0, \nu_T) = -  \inf_{\substack{\varphi, \eta}} \sup_{(t,x)} \bigg( & -\int_\Rr \varphi(0,x)dm_0(x) + \int_\Rr \varphi(T,x)d\nu_T(x) 
	\\
	&   + T \big( -\varphi_t + Q\eta + H(x,\varphi_x + \eta + u'_T)\big)\bigg), 
\end{align*}
where $(t,x) \in [0,T] \times \Rr$,  $\varphi \in \Lambda([0,T]\times\Rr)$ and $\eta \in C([0,T])$.
\end{theorem}
The previous result is proved in Section \ref{Sec:Duality} using Fenchel-Rockafellar's duality theorem.

Next,  in Section \ref{Sec:MFG Results}, we establish additional results for the MFG system \eqref{eq:MFG system} and \eqref{eq:initial terminal}. In particular, 
in Proposition \ref{prop:price Lipschitz}, 
we prove that $\varpi$ solving \eqref{eq:MFG system} and \eqref{eq:initial terminal} is Lipschitz continuous. This result was stated but not proved in \cite{gomes2018mean}. Here, 
we give the full details of the proof. 

Finally, in Section \ref{Sec:Proof Theorem 1.2}, we establish our main result,
which is summarized in the following theorem. 

\begin{theorem}\label{thm: Representation}
Let $(u,m,\varpi)$ solve \eqref{eq:MFG system} and \eqref{eq:initial terminal}. Suppose that Assumptions \ref{hyp: H convex}-
\ref{hyp: V-uT convex DS} hold. Then,
\begin{equation*}
	\int_{\Rr} \left(u(0, x) - u_T(x)\right) d m_0(x) - \int_0^T Q(t) \varpi(t) dt = 
	\inf_{\substack{\mu \in \mathcal{H}(m_0 )}}  \int_{\Omega} L(x,v) + v u_T'(x)d\mu(t,x,v).
\end{equation*}
\end{theorem}
In the previous theorem, the value of \eqref{eq:Min problem m0} is characterized by the solution of the MFG system \eqref{eq:MFG system} and \eqref{eq:initial terminal}. Although $m$ does not appear explicitly on the right-hand side of the previous expression, it determines the balance condition for the MFG. Notice that for this minimization problem, $u_T$ is fixed, whereas the terminal measure $\nu_T$ is varying (see Section \ref{Sec:Duality}).
%
%
%
%
%
%

\section{Assumptions}\label{Sec: Assupmtions}
Here, we present the main assumptions used in this paper. First, we consider the usual convexity assumption on the Hamiltonian, $H$, for which we require the strongest form of this property.

\begin{hyp}\label{hyp: H convex}
	For all $x\in\Rr$, the map $p\mapsto H(x,p)$ is 
 uniformly convex; that is, there exists a constant $\kappa >0$ such that $H_{pp}(x,p)\geq \kappa$ for all $(x,p)\in\Rr^2$.
\end{hyp}

The previous assumption guarantees not only convexity but also coercivity of $H$ in the $p$ variable (see \cite{bauschke2017convex}, Corollary 11.17). Hence, the Legendre transform of $H$, given by \eqref{eq:Legendre transform}, is well-defined, and it is convex and coercive in the second argument (\cite{cannarsa}, Theorem A.2.6). 

The following four assumptions are used in Section \ref{Sec:Duality} to establish duality results. The following growth conditions for $H$ and the regularity for $u_T$ and $Q$ are required
used when we apply Fenchel-Rockafellar's theorem.

\begin{hyp}\label{hyp: H growth}
There exists $\gamma_1 \geq  1$, $\gamma_2> 1$, a positive constant $C$, and non-negative constants $C_1$ and $C_2$ such that, for all $(x,p)\in \Rr^2$,
	\[
	\begin{cases}
	-C_2|x|^{\gamma_1} + \dfrac{|p|^{\gamma_2}}{C \gamma_2}-C \leq H(x,p) \leq -C_1 |x|^{\gamma_1} + \dfrac{|p|^{\gamma_2}C}{\gamma_2}+C,
	\\
	| H_x(x,p)| \leq C(|p|^{\gamma_2} +1), 
	\\
	| H_p(x,p)| \leq C(|p|^{\gamma_2-1} + 1).
	\end{cases}
	\]
\end{hyp}
\begin{remark}\label{remark:H growth}
Under Assumption \ref{hyp: H convex}, the Lagrangian, $L$, defined by \eqref{eq:Legendre transform}, satisfies (see \cite{cannarsa}, Theorem A.2.6)
\[
	v=-H_p(x,p) \mbox{ if and only if } p=-L_v(x,v).
\]
Furthermore, Assumption \ref{hyp: H growth} implies a growth condition on $L$; that is,
\begin{equation}\label{eq: L growth out of H}
C_1|x|^{\gamma_1}+\dfrac{|v|^{\gamma_2'}}{\gamma_2' C^{\gamma_2'/\gamma_2}}-C \leq L(x,v) \leq C_2|x|^{\gamma_1} + \dfrac{|v|^{\gamma_2'}C^{\gamma_2'/\gamma_2}}{\gamma_2'}+C,
\end{equation}
where $1/\gamma_2+1/\gamma_2'=1$. To see this, note that the first condition in Assumption \ref{hyp: H growth} bounds the Legendre transform of $H$ between the one of the functions
\[
	p\mapsto -C_1 |x|^{\gamma_1} + \dfrac{|p|^{\gamma_2}C}{\gamma_2}+C \quad \mbox{and} \quad p\mapsto -C_2|x|^{\gamma_1} + \dfrac{|p|^{\gamma_2}}{C \gamma_2}-C.
\]
Their transforms are the lower and upper bounds in \eqref{eq: L growth out of H}, respectively.
\end{remark}

\begin{hyp}\label{hyp: uT C1 Linfty}
		The terminal cost satisfies $u_T \in C^1(\Rr)$, and $|u'_T|\leq C$ for some $C>0$.
\end{hyp}

For the supply, we assume it is a smooth function of time.

\begin{hyp}\label{hyp:Q C infty}
	The supply function, $Q$, is $C^\infty([0,T])$. 
\end{hyp}

The existence of generalized measures minimizing our variational problem \eqref{eq:Min problem m0} relies on the moment estimates that we impose for the initial distribution (see Proposition \ref{prop: gamma1 moment mu}).

\begin{hyp}\label{hyp: m_0 moment}
	The initial density, $m_0$, is a probability measure in $\Rr$, and it has a finite absolute moment of order $\gamma>\gamma_1$; that is, 
	\begin{equation*}
\int_{\Rr}|x|^{\gamma} m_0(x)\dx < +\infty.
	\end{equation*} 
\end{hyp}

Following \cite{gomes2018mean}, we guarantee the solvability of \eqref{eq:MFG system} and \eqref{eq:initial terminal} by considering, together with Assumption \ref{hyp:Q C infty}, the following conditions.

\begin{hyp}\label{hyp: H separability DS}
	The Hamiltonian $H$ is separable; that is, 
	\begin{equation*}
		H(x,p)=\mathcal{H}(p)-V(x),
	\end{equation*} 
	where $V\in C^2(\Rr)$ is bounded from below and $|\mathcal{H}_{pp}|,|\mathcal{H}_{ppp}|\leq C$ for some constant $C>0$. 
\end{hyp}


\begin{remark}
Under the previous assumption, $L$, defined by \eqref{eq:Legendre transform}, is separable as well; that is
\[
	L(x,v)=\mathcal{L}(v) + V(x),
\]
where $\mathcal{L}$ is the Legendre transform of $\mathcal{H}$. Recalling that the Legendre transform is an involutive transformation, in case that $\mathcal{L}$ is uniformly convex, we have $\mathcal{L}_{vv}\geq \kappa'$ for some $\kappa'>0$. Hence, (\cite{cannarsa}, Corollary A. 2.7)
	\[
		\mathcal{H}_{pp}\leq \frac{1}{\kappa'}.
	\]
Furthermore, under Assumption \ref{hyp: H convex}, we obtain $\kappa<\mathcal{H}_{pp} \leq 1/\kappa'$. By abuse of notation, we set $\mathcal{H}=H$ and $\mathcal{L}=L$ when Assumption \ref{hyp: H separability DS} holds.
\end{remark}
\begin{hyp}\label{hyp: V-uT Lipschitz 2nd D bounded DS}
	The potential $V$, the terminal cost $u_T$, the initial density function $m_0$ are $C^2(\Rr)$ functions and  $V$, $u_T$ are  globally Lipschitz. Furthermore,  there exists a constant $C>0$ such that
	\[
		|V''|\leq C, \quad |u_T''|\leq C, \quad |m_0''|\leq C.
	\]
\end{hyp}

The following condition guarantees the uniqueness of solutions of \eqref{eq:MFG system} and \eqref{eq:initial terminal}.

\begin{hyp}\label{hyp: V-uT convex DS}
	The potential $V$ and the terminal cost $u_T$ are convex.
\end{hyp}

\begin{remark}
Assume the Hamiltonian, $H$, satisfies Assumption \ref{hyp: H separability DS}, with a potential, $V$, satisfying Assumption \ref{hyp: V-uT Lipschitz 2nd D bounded DS}. For Assumption \ref{hyp: H growth} to hold, $V$ has to satisfy $C \geq \mbox{Lip}(V)$ and the growth condition 
\begin{equation}\label{V}
	C_1|x|^{\gamma_1}  -K \leq V(x) \leq C_2|x|^{\gamma_1}+K
\end{equation}
for some $K>0$, whereas $\mathcal{H}$ has to satisfy $| \mathcal{H}_p(p)| \leq C(|p|^{\gamma_2-1} + 1)$ and the growth condition 
	\[
		\dfrac{|p|^{\gamma_2}}{C \gamma_2}-C \leq H(p) \leq \dfrac{|p|^{\gamma_2}C}{\gamma_2}+C.
	\]
For instance, the  Hamiltonian 
\begin{equation*}
H(x,p)=\left( 1+|p|^2\right)^{\frac{\gamma_2}{2}}-V(x)
\end{equation*} 
satisfies all the assumptions above if $V$ is a globally Lipschitz function that satisfies \eqref{V}.
\end{remark}

\section{Duality results}\label{Sec:Duality}

This section considers generalized holonomic measures for time-dependent problems in $\Rr$ that are compatible with the integral constraint imposed by the balance condition. We use this formulation to prove Theorem \ref{thm:convex duality h} and for the proof of Theorem \ref{thm: Representation} in Section \ref{Sec:Proof Theorem 1.2}. 

Fix $T>0$. For $\gamma_1\geq 1$ and $\gamma_2>1$ (see Assumption \ref{hyp: H growth}), let $\zeta=(\zeta_1,\zeta_2)$, where
\begin{equation}\label{eq: zetas selection}
0<\zeta_1 \leq \gamma_1, \quad \mbox{and}
\quad 1< \zeta_2 < \gamma'_2.
\end{equation}
Let $\Omega =[0,T] \times\Rr\times\Rr $. Let $\mathcal{R}(\Omega)$ be the set of signed Radon measures on $\Omega$, $\mathcal{R}^+(\Omega)$ be the subset of non-negative elements of $\mathcal{R}(\Omega)$ (\cite{FollandRealAnalysis}, page 212 and 222 or \cite{evansgariepy2015}, 
Definition 1.9), and $\Pp(\Rr)$ be the set of probability measures on $\Rr$. We define
\begin{align}\label{eq:holonomy converge}
\mathcal{H}_1 =\left\{ \mu \in \mathcal{R}^+(\Omega):~\int_{\Omega} (|x|^{\zeta_1}+|v|^{\zeta_2} )d\mu(t,x,v)< \infty \right\}.
\end{align}
This set is determined by the growth conditions for the Hamiltonian, as in Assumption \ref{hyp: H growth}. Next, let 
\begin{equation*}
\Lambda([0,T]\times\Rr)=\left\{ \varphi \in C^1([0,T]\times \Rr):~ \varphi_t,~ \varphi_x \in L^\infty([0,T]\times \Rr) \right\}.
\end{equation*}
Notice that elements of $\Lambda([0,T]\times\Rr)$ are globally Lipschitz continuous functions. This set corresponds to the set of test functions for the holonomy condition, which we define next. Given $m_0,~\nu \in \Pp(\Rr)$, let 

\begin{equation}\label{eq:holonomy equality}
\renewcommand{\arraystretch}{1.4}
\mathcal{H}_2(m_0,\nu)=\left\{ \mu \in \mathcal{R}^+(\Omega):~
\begin{matrix}
\int_{\Omega} \varphi_t (t, x) + v \varphi_x(t, x) d\mu (t,x,v) \\
 =\int_{\Rr} \varphi (T, x) d\nu - \int_{\Rr} \varphi (0, x) dm_0
 \end{matrix} ~ \forall \varphi \in \Lambda([0,T]\times\Rr)  \right\}.
\end{equation}
As mentioned in the Introduction, we refer to the condition defining the set $\mathcal{H}_2(m_0,\nu)$ as the holonomy condition. For a given $\nu \in \Pp(\Rr)$, the set $\mathcal{H}_2(m_0,\nu)$ may be empty. Nevertheless, as we show in Remark \ref{Remark:H(m0) not empty}, there are probability measures satisfying $\mathcal{H}_2(m_0,\nu) \neq \emptyset$. In case $m_0$ satisfies a moment hypothesis (see Assumption \ref{hyp: m_0 moment}), the identity that defines the holonomy condition is well-defined even if the terms are not finite.

Corresponding to the balance condition in \eqref{eq:MFG system}, we set 
\begin{align}\label{eq:holonomy balance}
\mathcal{H}_3 =\left\{ \mu \in \mathcal{R}^+(\Omega):~\int_{\Omega} \eta(t) (v-Q(t)) d\mu (t,x,v) =0, \quad \forall \eta \in C([0,T])\right\}.
\end{align}
Finally, we define
\[
	\mathcal{H} (m_0,\nu):=\mathcal{H}_1 \cap \mathcal{H}_2(m_0,\nu) \cap \mathcal{H}_3, \quad \mbox{and} \quad
	\mathcal{H} (m_0)=\bigcup_{\nu \in \Pp(\Rr)} \mathcal{H} (m_0,\nu).
\]
\begin{remark}
For any $\mu\in \mathcal{H} (m_0)$ there exists a unique $\nu \in \Pp(\Rr)$ such that $\mu \in \mathcal{H}(m_0,\nu)$. To see this, let $\mu \in \mathcal{H}(m_0,\nu)\cap\mathcal{H}(m_0,\tilde{\nu})$. Let $\varphi \in C^1_c(\Rr)$. Then,  $\varphi \in\Lambda([0,T]\times \Rr)$, and \eqref{eq:holonomy equality} holds for both $\nu$ and $\tilde{\nu}$, from which we obtain
\[
	\int_{\Rr} \varphi(x)d\nu(x) - \int_{\Rr} \varphi(x) dm_0(x) = \int_{\Rr} \varphi (x) d\tilde{\nu}(x) - \int_{\Rr} \varphi (x) dm_0(x);
\]
that is, $\int_{\Rr}\varphi d\nu = \int_{\Rr} \varphi d\tilde{\nu}$ for any $\varphi \in C^1_c(\Rr)$. Hence, $\nu  = \tilde{\nu}$ (\cite{FollandRealAnalysis}, Theorem 7.2). We denote the unique measure $\nu$ such that $\mu \in \mathcal{H}(m_0,\nu)$ as $\nu^\mu$.
\end{remark}

\begin{remark}\label{Remark:H(m0) not empty}
If Assumptions \ref{hyp:Q C infty} and \ref{hyp: m_0 moment} hold, 
$\mathcal{H}(m_0)$ is not empty. To see this, let $\overline{X}(t)=\int_0^t Q(s) ds$. Define $\nu \in \Pp(\Rr)$ by
\begin{equation}\label{def: nu}
\int_{\Rr} f(x)d\nu(x) = \int_{\Rr} f(x+\overline{X}(T))dm_0(x)
\end{equation}
for all $f \in C_c(\Rr)$, and define $\mu \in \mathcal{R}^+(\Omega)$ by
\begin{equation}\label{def-mu}
\int_{\Omega} \psi(t,x,v) d\mu(t,x,v) = \int_0^T \int_{\Rr} \psi(t,x+\overline{X}(t),Q(t))dm_0(x) dt
\end{equation}
for all $\psi  \in C_c(\Omega)$. Next, we use the following cut-off function 
\begin{equation*}
	\theta(x)=\begin{cases}
		1\quad &x\in (-1,1)\\
		\vartheta(x)\quad &(-2,2)\setminus(-1,1)\\	
		0\quad &x\in\Rr\setminus	
		(-2,2),
	\end{cases}
\end{equation*}
where $\vartheta$ is chosen such that $\theta\in C^1(\Rr)$, $0\leq\theta\leq 1$, and $||\theta||_{C^1(\Rr)}\leq c$.  Let 
\[
	h_n(x,v)=\theta\left( \frac{2x}{n}\right)\theta\left( \frac{2v}{n}\right) \left(|x|^{\zeta_1}+|v|^{\zeta_2}\right) \mbox{  and  } g_n=\mathds{1}_n(x)\mathds{1}_n(v) \left(|x|^{\zeta_1}+|v|^{\zeta_2}\right),
\]
where $\mathds{1}_n$ is the characteristic function of the interval $[-n,n]$.  $g_n$ is a sequence of measurable functions that satisfy $0 \leq g_n(x,v) \leq |x|^{\zeta_1}+|v|^{\zeta_2}$ and $g_{\frac{n}{2}}(x,v)\to |x|^{\zeta_1}+|v|^{\zeta_2}$ pointwise for $(x,y)\in\Rr^2$. Although the functions $g_n$ are not continuous, they are Borel-measurable, and hence their integral w.r.t. $\mu$ is well-defined. Note that $g_{\frac{n}{2}}(x,v)\leq h_n(x,v)\leq g_n(x,v)$  and $h_n\in C_c(\Omega)$. Then, 
\begin{align*}
	\int_{\Omega} g_{\frac{n}{2}}(x,v) d\mu(t,x,v)&\leq \int_{\Omega} h_n(x,v) d\mu(t,x,v) 
	\\
	&=\int_{0}^T \int_{\Rr}h_n(x+\overline{X}(t),Q(t)) dm_0(x) dt 	\\& \leq\int_{0}^T \int_{\Rr}g_n(x+\overline{X}(t),Q(t)) dm_0(x) dt 
	\\
	& \leq \int_0^T \int_{\Rr}|x+\overline{X}(t)|^{\zeta_1}+ |Q (t)|^{\zeta_2} dm_0(x) dt 
	\\
	& \leq \int_0^T \int_\Rr 2^{\zeta_1-1} \left( |x|^{\zeta_1} + |\overline{X}(t)|^{\zeta_1} \right) + |Q(t)|^{\zeta_2} dm_0(x) dt
	\\
	& = 2^{\zeta_1-1}\left( T \int_\Rr |x|^{\zeta_1} dm_0(x) + \|\overline{X}\|^{\zeta_1}_{L^{\zeta_1}([0,T])}\right)+\|Q\|^{\zeta_2}_{L^{\zeta_2}([0,T])}
	\\
	& \leq  2^{\zeta_1-1} T\left( \int_\Rr |x|^{\zeta_1} dm_0(x) +  T^{\zeta_1}\|Q\|_\infty^{\zeta_1} \right)+T \|Q\|^{\zeta_2}_{\infty}
	\\
	& = C(\zeta_1,\zeta_2, T, m_0, Q),
\end{align*}
where $C(\zeta_1,\zeta_2, T, m_0, Q)$ is finite by Assumptions \ref{hyp:Q C infty} and \ref{hyp: m_0 moment}. Using the previous inequality and 
the Monotone Convergence Theorem, we conclude that $\mu$ satisfies \eqref{eq:holonomy converge}.
Therefore,
 for any $\varphi\in\Lambda([0,T] \times \Rr)$,  we have
\begin{equation}\label{conver-int}
\begin{split}
\int_{\Omega} |\varphi_t (t,x)| + |v| |\varphi_x(t,x)| d\mu (t,x,v)<\infty,\\
\int_{\Rr} |\varphi(T,x+\overline{X}(T))| dm_0(x)+\int_{\Rr} |\varphi(0,x)|dm_0(x)<\infty.
\end{split}
\end{equation}
Denote $\overline{M}=\max\limits_{t\in[0,T]}|\overline{X}(t)|$ and let
$\phi^{n}(t,x,v)=\varphi^n(t,x)\theta\left(\frac{v}{n} \right)$, where $\varphi^n(t,x)=\varphi(t,x)\theta\left(\frac{x}{n} \right)$. Because  $\varphi\in\Lambda([0,T] \times \Rr)$, from the definitions of $\phi^{n}$, $\varphi^{n}$, and $\theta$, we have
\begin{gather}\label{esti}
	\max_{(t,x)\in[0,T]\times\Rr}|\varphi^{n}(t,x)|=\max_{(t,x)\in[0,T]\times\Rr}\left| \varphi(t,x)\theta\left( \tfrac{x}{n}\right) \right|\leq C n, 
	\\
	\max_{(t,x,v)\in\Omega}|\phi^{n}_t(t,x,v)|\leq	\max_{(t,x)\in\Omega}|\varphi^{n}_t(t,x)|=\max_{(t,x)\in[0,T]\times\Rr}\left| \varphi_t(t,x)\theta\left( \tfrac{x}{n}\right) \right| \leq C,  \nonumber
	\\
||\phi^{n}_x||_{C(\Omega)}\leq	\max_{(t,x)\in[0,T]\times\Rr}|\varphi^{n}_x(t,x)|=\max_{(t,x)\in[0,T]\times\Rr}\left| \varphi_x(t,x)\theta\left( \tfrac{x}{n}\right)+ \tfrac{1}{n}\varphi(t,x)\theta^\prime\left( \tfrac{x}{n}\right) \right| \leq C. \nonumber
\end{gather}
Relying on  these estimates from Assumption \ref{hyp: m_0 moment}, we have
\begin{align}\label{o1}
	&\left| \int_{0}^{T}\int_{n<|x|<2n}\int_{n<|v|<2n} \phi^{n}_t  + v \phi^{n}_x d\mu (t,x,v) \right| \nonumber
	\\
	&\leq \int_{0}^{T}\int_{n<|x|<2n}\int_{n<|v|<2n} C(|v| +1)d\mu (t,x,v)	\nonumber 
	\\
	&\leq \int_{\Omega} C(|v| +1)\left( \theta\left( \frac{2x}{n}-3\right)+\theta\left(  \frac{2x}{n}+3\right)\right) \left( \theta\left( \frac{2v}{n}-3\right)+\theta\left(  \frac{2v}{n}+3\right)\right) d\mu (t,x,v)	 \nonumber
	\\
	&\leq \int_{0}^{T} \int_{\Rr}2C(|Q(t)| +1) \left(\theta\left( \frac{2(x+\overline{X}(t))}{n}-3\right)+\theta\left(  \frac{2(x+\overline{X}(t))}{n}+3\right)\right)dm_0(x) \nonumber
	\\
	&\leq TC\int_{\frac{n}{2}-\overline{M}<|x|<\frac{5}{2}n+\overline{M}}dm_0(x)=o(1).
\end{align}
Furthermore, Assumption \ref{hyp: m_0 moment} with \eqref{esti} implies that 
\begin{align}
	\label{id}
	& \left|\left(  \int_{n}^{2n-\overline{X}(T)}+ \int_{-2n-\overline{X}(T)}^{-n}\right) \varphi^n(T,x+\overline{X}(T)) dm_0(x)\right|	 \nonumber	\\	 \nonumber	
	 &\leq  \int_{n-\overline{M}<|x|<2n+\overline{M}} \left| \varphi^n(T,x+\overline{X}(T)) \right| dm_0(x)\leq C\int_{n-\overline{M}<|x|<2n+\overline{M}} |x| dm_0(x)=o(1),
	 \nonumber	
	\\
	&\left| \int_{n<|x|<2n} \varphi^n(0,x) dm_0(x)\right|\leq C\int_{n<|x|<2n} |x| dm_0(x)=o(1),
\end{align}
where $o(1)\to0$ when $n\to\infty$.
Note that  $\phi^{n}\in C_c(\Omega)$ and $\supp(\phi^{n})=[0,T]\times[-2n,2n]^2$.
 Consequently, for all $n\geq n_0$, where $n_0$ satisfies $\frac{\|Q\|_{\infty}}{n_0}\leq 1$,   by \eqref{def-mu}, we have
\begin{align}
\label{iden-1}
		& \int_{0}^{T}\int_{-2n}^{2n}\int_{-2n}^{2n}\phi^{n}_t (t,x,v) + v \phi^{n}_x(t,x,v) d\mu (t,x,v) \nonumber	
		\\
		&= \int_0^T\int_{\Rr}  \phi^n_t (t,x+\overline{X}(t),Q(t)) + Q(t) \phi^{n}_x(t,x+\overline{X}(t),Q(t))dm_0(x)\dt \nonumber
		\\
		&= \int_0^T\int_{\Rr}\theta\left(\frac{Q(t)}{n} \right)\left(  \varphi^n_t (t,x+\overline{X}(t))+ Q(t) \varphi^n(t,x+\overline{X}(t)) \right)dm_0(x)dt \nonumber
		\\
		&=\int_{\Rr} \int_0^T \frac{d}{dt}\varphi^n(t,x+\overline{X}(t)) dm_0(x)dt \nonumber
		\\
		&= \int_{-2n-\overline{X}(T)}^{2n-\overline{X}(T)} \varphi^n(T,x+\overline{X}(T)) dm_0(x)-\int_{-2n}^{2n} \varphi^n(0,x)dm_0(x).
\end{align}
On the other hand, \eqref{o1} and \eqref{id}, yield
 \begin{align*}
 	&\int_{0}^{T}\int_{n<|x|<2n}\int_{n<|x|<2n} \phi^{n}_t (t,x) + v \phi^{n}_x(t,x) d\mu (t,x,v) 
 	\\
 	&- \left(  \int_{n}^{2n-\overline{X}(T)}+ \int_{-2n-\overline{X}(T)}^{-n}\right)\varphi^{n}(T,x+\overline{X}(T)) dm_0(x)+\int_{n<|x|<2n} \varphi^{n}(0,x)dm_0(x) =o(1).
\end{align*}
Therefore, \eqref{iden-1} with  the definitions of $\phi^{n}$, $\varphi^{n}$, $\theta$  implies 
\begin{equation}\label{iden-2}
	\begin{split}
		\int_{0}^{T}\int_{-n}^{n}\int_{-n}^{n} \varphi_t (t,x) &+ v \varphi_x(t,x) d\mu (t,x,v) \\
		&= \int_{-n}^{n} \varphi(T,x+\overline{X}(T)) dm_0(x)-\int_{-n}^{n} \varphi(0,x)dm_0(x)+o(1).
	\end{split}
\end{equation}
With similar arguments, by using \eqref{def: nu}, we  prove that 
\begin{equation}\label{iden-3}
	\begin{split}
	\int_{-n}^{n} \varphi(T,x+\overline{X}(T)) dm_0(x)&-\int_{-n}^{n} \varphi(0,x)dm_0(x) \\
		&=   \int_{-n}^{n} \varphi(T,x) d\nu(x)- \int_{-n}^{n}\varphi(0,x) dm_0(x)+o(1).
	\end{split}
\end{equation}
Combining  \eqref{iden-2} and \eqref{iden-3}, we obtain 
  \begin{equation*}
 	\begin{split}
 		\int_{0}^{T}\int_{-n}^{n}\int_{-n}^{n} \varphi_t (t,x) &+ v \varphi_x(t,x) d\mu (t,x,v) \\
 		&=   \int_{-n}^{n} \varphi(T,x) d\nu(x)- \int_{-n}^{n}\varphi(0,x) dm_0(x) +o(1).
 	\end{split}
 \end{equation*}
Letting $n\to \infty$ in the preceding identity and using \eqref{conver-int}, we conclude that $\mu$ satisfies \eqref{eq:holonomy equality}. 

Lastly, proceeding as before, we prove that $\mu$ verifies \eqref{eq:holonomy balance}. Hence, $\mu \in \mathcal{H}(m_0,\nu)$ and, therefore, $\mathcal{H}(m_0) \neq \emptyset$.
\end{remark}

The minimization in \eqref{eq:Min problem m0} is an infinite-dimensional optimization problem. To study the connection between solutions of \eqref{eq:MFG system} and the dual problem of \eqref{eq:Min problem m0}, we compute the dual problem using Fenchel-Rockafellar's theorem (\cite{Villanithebook},  Theorem 1.9.):
\begin{theorem}\label{Thm: Fenchel-Rockafellar}
Let $E$ be a normed vector space and let $E^*$ be its topological dual space. Let $f$ and $g$ be convex functions on $E$ with values in $\Rr \cup \{+\infty\}$. Denote by $f^*$ and $g^*$ the Legendre-Fenchel transforms of $f$ and $g$, respectively, defined by
\[
	f^* (x^*)=\sup_{ x \in E} \left( \left\langle x^*,x\right\rangle - f(x) \right), \quad g^* (x^*)=\sup_{ x \in E} \left( \left\langle x^*,x\right\rangle - g(x) \right).
\]
Assume there exists $x_0 \in E$ such that $f(x_0), ~g(x_0) < +\infty$, and $f$ is continuous at $x_0$. Then
\begin{equation}\label{eq: Aux FenchelRock statement}
	\inf_{x \in E} f(x) + g(x) = \max_{y \in E^*} - f^*(-y) - g^*(y).
\end{equation}
\end{theorem}
In the previous result, it is part of the theorem that the supremum in the right-hand side of \eqref{eq: Aux FenchelRock statement} is a maximum. 

Now, we introduce the definitions we need to apply Theorem \ref{Thm: Fenchel-Rockafellar}. Recall that $\Omega =[0,T] \times\Rr\times\Rr $, and let $\zeta=(\zeta_1,\zeta_2)$ according to \eqref{eq: zetas selection}. Consider the normed vector space
\begin{align}\label{def: Czeta}
	C_\zeta(\Omega) :=& \left\{ \phi \in C(\Omega):~ \|\phi\|_{\zeta}:= \sup_{\Omega} \left| \frac{ \phi(t,x,v)}{1+|x|^{\zeta_1} + |v|^{\zeta_2}} \right|< \infty, \right.
	\\
	& \left. ~ \lim_{|x|,|v| \to \infty} \frac{\phi(t,x,v)}{1+|x|^{\zeta_1} + |v|^{\zeta_2}} = 0 ~ \mbox{uniformly for } t \in [0,T]\right\}.\nonumber
\end{align}

\begin{remark}\label{remark:Czeta dual}
Let $\zeta$ satisfy \eqref{eq: zetas selection}. The dual of $\left(C_\zeta(\Omega),\|\cdot\|_\zeta\right)$ is 
\begin{align*}
\mathcal{U}^\zeta=\left\{ \mu \in \mathcal{R}(\Omega):~ \int_{\Omega} (1+|x|^{\zeta_1} + |v|^{\zeta_2})d|\mu|(t,x,v) < \infty \right\}.
\end{align*}
To see this, let 
\[
C_0(\Omega):=\left\{ \psi \in C(\Omega):~ \lim_{|x|,|v|\to \infty} \psi(t,x,v) = 0, ~ \mbox{uniformly for }t \in [0,T] \right\}.
\]
From the Riesz Representation Theorem (\cite{FollandRealAnalysis}, Theorem 7.17), we have that 
\begin{equation}\label{eq: RieszC0}
C_0(\Omega)^*\mbox{ and }\mathcal{R}(\Omega) \mbox{ are isomorphic}.
\end{equation}
Define $\Phi: C_0(\Omega) \to  C_\zeta(\Omega)$ by $\Phi(\psi)=\phi:=(1+|x|^{\zeta_1} + |v|^{\zeta_2}) \psi$. Then $\Phi$ is a linear isometry since $\|\Phi(\psi)\|_\zeta=\|\psi\|_{\infty}$. Now, given $f \in C_\zeta(\Omega)^*$, define $F\in C_0(\Omega)^*$ by $F=f\circ \Phi$. Using \eqref{eq: RieszC0}, there exists $\tilde{\mu}\in \mathcal{R}(\Omega)$ such that
\[
\langle F,\psi\rangle  = \int_\Omega \psi(t,x,v) ~d \tilde{\mu}(t,x,v)
\]
for all $\psi \in C_0(\Omega)$. 
Given $\phi \in C_\zeta(\Omega)$, let $\psi=\Phi^{-1} (\phi)=\frac{\phi}{1+|x|^{\zeta_1}+|v|^{\zeta_2}} \in C_0(\Omega)$. Then
\[
\langle f, \phi \rangle = \langle F,\psi\rangle=\int_\Omega \phi(t,x,v) ~\frac{d \tilde{\mu}(t,x,v)}{1+|x|^{\zeta_1} + |v|^{\zeta_2}}.
\]	
Hence, because $(x,v)\mapsto \frac{1}{1+|x|^{\zeta_1} + |v|^{\zeta_2}}$ is continuous and bounded, the measure $d\mu(t,x,v):= \frac{1}{1+|x|^{\zeta_1} + |v|^{\zeta_2}}d \tilde{\mu}(t,x,v)$ is a Borel measure finite on compact sets. Therefore (\cite{FollandRealAnalysis}, Theorem 7.8), $\mu$ is a Radon measure on $\Omega$. Notice that any Hahn, and therefore, Jordan decomposition of $\tilde{\mu}$ (\cite{FollandRealAnalysis}, Theorem 3.4) provides a corresponding decomposition for $\mu$, from which we obtain that $d|\mu|= \frac{1}{1+|x|^{\zeta_1} + |v|^{\zeta_2}} d|\tilde{\mu}|$. Therefore,
\[
\int_\Omega (1+|x|^{\zeta_1} + |v|^{\zeta_2} )d|\mu|(t,x,v) =\int_\Omega d|\tilde{\mu}|(t,x,v) < \infty.
\]
On the other hand, any $\mu \in \mathcal{U}^\zeta$ defines a linear map on $C_\zeta(\Omega)$ by
\[
	\phi \mapsto \int_{\Omega} \phi(t,x,v) d|\mu|.
\]
From the following inequality 
\[
	\left| \int_{\Omega}  (1+|x|^{\zeta_1} + |v|^{\zeta_2}) \dfrac{\phi(t,x,v)}{1+|x|^{\zeta_1} + |v|^{\zeta_2}} d|\mu| \right| \leq \|\phi\|_{\zeta} \int_{\Omega}  (1+|x|^{\zeta_1} + |v|^{\zeta_2})  d|\mu|,
\]	
we see that this linear map is also bounded.
Hence, we conclude that $C_\zeta(\Omega)^*$ and $\mathcal{U}^\zeta(\Omega)$ are isomorphic. It can be proved that they are isometrically isomorphic (see \cite{FollandRealAnalysis}, Theorem 7.17).
\end{remark}

Define (see Remark \ref{remark:Czeta dual})
\begin{align}\label{def: U1}
	\mathcal{U}_1=\left\{ \mu \in \mathcal{U}^\zeta:~ \mu \geq 0,~ \int_{\Omega} d \mu = T \right\}.
\end{align}
Notice that $\mathcal{U}_1$ is the set of non-negative Radon measures that satisfy \eqref{eq:holonomy converge} and for which \eqref{eq:holonomy equality} holds for $\varphi(t,x)=t$.  Now, we define an operator related to the left-hand side of \eqref{eq:holonomy equality}. Take $v \in \Rr$. Define, $	A^v:  ~ C^1([0,T] \times \Rr)  \to C_\zeta(\Omega)$  by
\begin{equation*}
	\varphi \mapsto A^v \varphi= -\varphi_t - v \varphi_x.
\end{equation*}
Indeed, because $\varphi_t,~ \varphi_x \in C([0,T]\times \Rr)$, and 
\[
	\frac{|A^v \varphi|}{1+|x|^{\zeta_1}+|v|^{\zeta_2}} \leq \|\varphi\|_{C^1} \frac{1+|v|}{1+|x|^{\zeta_1}+|v|^{\zeta_2}}\leq \|\varphi\|_{C^1} \left( 1+ \sup_{v\in \Rr} \frac{|v|}{1+|v|^{\zeta_2}}\right) \leq C \|\varphi\|_{C^1},
\]
we have $A^v \varphi \in C_\zeta(\Omega)$ and $A^v$ is bounded. Therefore, $A^v$ is a linear and bounded map. We use this map to define the following sets. Let $\mathcal{C}\subset C_\zeta(\Omega)$ be the closed subspace
\begin{align}\label{def: C set}
	\mathcal{C}=\mbox{cl}_{\| \cdot \|_{\zeta}} \left\{ \phi \in C_\zeta(\Omega):~ \phi(t,x,v)=A^v \varphi(t,x)-(v-Q(t))\eta(t)~  \mbox{ for some } \right.
	\\
	\left. \varphi \in \Lambda([0,T]\times \Rr),~ \eta \in C([0,T])\right\}, \nonumber
\end{align}
where $Q$ satisfies Assumption \ref{hyp:Q C infty}, and $\mbox{cl}_{\| \cdot \|_{\zeta}}$ denotes the closure with respect to $\| \cdot \|_{\zeta}$. Notice that $\mathcal{C}$ is convex because $A^v$ is linear. 

Given a linear and bounded operator $B:C([0,T]\times \Rr) \to \Rr$, let 
\begin{align}\label{def: U2}
	\mathcal{U}_2(B)=\mbox{cl}_{\mbox{{\tiny weak}}} \left\{ \mu \in \mathcal{U}^\zeta:~ \int_{\Omega} A^v \varphi d\mu(t,x,v) = B \varphi,~ \forall \varphi \in \Lambda([0,T]\times \Rr)\right\},
\end{align}
where $\mbox{cl}_{\mbox{{\tiny weak}}}$ denotes the closure with respect to weak convergence of measures (\cite{evansgariepy2015}, 
Definition 1.31.). The choice of the operator $B$ determines whether $\mathcal{U}_2(B)\neq\emptyset$. For instance, given $\nu_T \in \mathcal{P}(\Rr)$, for the operators
\begin{equation}\label{def: B operator}
B \varphi = \int_\Rr \varphi(0,x)dm_0(x) - \int_\Rr \varphi(T,x)d\nu_T(x),
\end{equation}
and $A^v$ as before, \eqref{def: U2} corresponds to \eqref{eq:holonomy equality}, and Remark \ref{Remark:H(m0) not empty} shows that $\mathcal{U}_2(B)\neq\emptyset$. Analogously, (see \eqref{eq:holonomy balance}) we define
\begin{align}\label{def: U3}
	\mathcal{U}_3 = \mbox{cl}_{\mbox{{\tiny weak}}}\left\{ \mu \in \mathcal{U}^\zeta:~ \int_{\Omega} \eta(t) (v-Q(t)) ~ d\mu(t,x,v) =0, ~ \forall \eta \in C([0,T]) \right\}.
\end{align}

\begin{remark}\label{remark: H equals Us intersection}
Let $B$ as in \eqref{def: B operator}. If $m_0,\nu\in\Pp(\Rr)$ are such that $\mathcal{H}(m_0,\nu)\neq \emptyset$, then 
\[
	\mathcal{H}(m_0,\nu) = \mathcal{U}_1\cap\mathcal{U}_2(B)\cap\mathcal{U}_3.
\]
To see this, notice that \eqref{def: U1}, \eqref{def: U2}, and \eqref{def: U3} imply $\mathcal{U}_1\cap\mathcal{U}_2(B)\cap\mathcal{U}_3 \subset \mathcal{H}(m_0,\nu)$. For the opposite inclusion, let $\mu\in \mathcal{H}(m_0,\nu)$ and let $A=\left\{ (t,x,v) \in \Omega:~ 1 < |x|^{\zeta_1} + |v|^{\zeta_2}\right\}$. We have that $|\mu|=\mu$ because $\mu \geq 0$. Writing $\int_{\Omega}  d\mu =\int_{A}  d\mu + \int_{A^c}  d\mu$, where $A^c$ denotes the complement of the set $A$, we see that
\[
		\int_{A}  d\mu \leq \int_{A} |x|^{\zeta_1}+|v|^{\zeta_2} d\mu \leq \int_{\Omega} |x|^{\zeta_1}+|v|^{\zeta_2} d\mu<\infty,
\]
and $\int_{A^c}  d\mu$ is finite because $A^c$ is compact and $\mu$ is a Radon measure. Hence, $\mu \in \mathcal{U}^\zeta$. Moreover, since $\mu$ satisfies \eqref{eq:holonomy equality} and \eqref{eq:holonomy balance},  we have that $\mu \in \mathcal{U}_2(B)\cap\mathcal{U}_3$. Using $\varphi(t,x)=t$ in  \eqref{eq:holonomy equality}, we obtain that $\mu \in \mathcal{U}_1$.  Therefore $\mu \in \mathcal{U}_1\cap\mathcal{U}_2(B)\cap\mathcal{U}_3$. 
\end{remark}

Now, we introduce the functionals we will use in the context of the Fenchel-Rockafellar theorem. Define $f: C_\zeta(\Omega) \to \Rr\cup\{+\infty\}$ by
\begin{align}\label{def: f map}
	f(\phi)=T \sup_{(t,x,v)\in\Omega} \left(\phi(t,x,v) - L(x,v)-vu'_T(x)\right).
\end{align}
Since $f$ is the supremum of affine functions, $f$ is convex. The following result proves continuity for this map.
\begin{lemma}\label{lem: f continuous}
Let $\zeta$ satisfy \eqref{eq: zetas selection}. Under Assumptions \ref{hyp: H convex}-
\ref{hyp: uT C1 Linfty}, the map $f$ is continuous on $C_\zeta(\Omega)$.
\end{lemma}
\begin{proof}
Let $(\phi_n)_{n\in\Nn},\phi \in C_\zeta(\Omega)$ be such that $\lim_{n} \|\phi_n-\phi\|_\zeta=0$. The first condition in \eqref{def: Czeta} and the convergence of $\phi_n$ guarantees the existence of $C>0$ such that $\|\phi_n\|_\zeta,\|\phi\|_\zeta \leq C$ for all $n$; that is,
\[
	|\phi_n(t,x,v)|,|\phi(t,x,v)| \leq C (1+|x|^{\zeta_1}+|v|^{\zeta_2}) \quad \mbox{for all } (t,x,v)\in \Omega,~ n\in \Nn.
\]
Let $\alpha> C$. By Assumption \ref{hyp: H growth}, using \eqref{eq: L growth out of H} (see Remark \ref{remark:H growth}), we have
\[
 C_1|x|^{\gamma_1} + \dfrac{|v|^{\gamma_2'}}{\gamma_2' C^{\gamma_2' / \gamma_2}}-C \leq L(x,v), \quad \text{for all} \quad(x,v)\in \Rr^2.
\]
Adding the term $v u'_T(x)$ to both sides of the previous inequality, we get
\[
	\dfrac{1}{\gamma_2' C^{\gamma_2' / \gamma_2}}\left( \dfrac{\gamma_2' C^{\gamma_2' / \gamma_2}C_1 |x|^{\gamma_1} + |v|^{\gamma_2'} +v u_T'(x)-\gamma_2' C^{\gamma_2' / \gamma_2} C }{1+|x|^{\zeta_1}+|v|^{\zeta_2}}\right)  \leq \frac{L(x,v)+v u'_T(x)}{1+|x|^{\zeta_1}+|v|^{\zeta_2}}
\]
for all $(x,v)\in \Rr^2$. By Assumption \ref{hyp: uT C1 Linfty}, $u'_T$ is bounded. Hence, according to \eqref{eq: zetas selection}, the left-hand side of the previous expression goes to $+\infty$  when $|x|,|v| \to +\infty$. Hence, we can find $r>0$ such that $|x|,|v| \geq r$ implies 
\[
	-\frac{L(x,v)+v u'_T(x)}{1+|x|^{\zeta_1}+|v|^{\zeta_2}} \leq -\alpha.
\]
Let $(x,v) \in ([-r,r]^2)^c$, where $A^c$ denotes the complement of the set $A$, and let $t \in [0,T]$. Using the previous bound, we have
\begin{align*}
	\phi_n(t,x,v) - L(x,v) -v u'_T(x) & \leq \phi_n(t,x,v) - \alpha (1+|x|^{\zeta_1}+|v|^{\zeta_2})  
	\\
	& \leq (C-\alpha)(1+|x|^{\zeta_1}+|v|^{\zeta_2})
	\\
	& < 0,
\end{align*}
for $n\in \Nn$. Hence,
\[
	f(\phi_n) = T \sup_{(t,x,v)\in [0,T] \times [-r,r]^2} \left(\phi_n(t,x,v) - L(x,v)-vu'_T(x)\right), 
\]	
and the same holds for $\phi$. Because the convergence on $C_\zeta(\Omega)$ implies uniform convergence on $[0,T]\times [-r,r]^2$, we obtain
\[
	f(\phi_n) \to f(\phi). \qedhere
\]	
\end{proof}
\begin{proposition}\label{Lemma5.1. alike}
Let $\zeta$ satisfy \eqref{eq: zetas selection}. Suppose that Assumptions \ref{hyp: H convex}-
\ref{hyp: uT C1 Linfty} hold. Let $\mu \in \mathcal{U}^\zeta$. If $\mu \ngeq 0$ then $f^*(\mu) = +\infty$.
\end{proposition}
\begin{proof}
Let $\mu \in \mathcal{U}^\zeta$ be such that $\mu \ngeq 0$. Regarding $\mu$ as a linear map, by Remark \ref{remark:Czeta dual}, there exists $\phi \in C_\zeta(\Omega)$ such that $0\leq \phi$ and $\int_{\Omega} \phi(t,x,v)d\mu <0$. Let $\phi_n=-n \phi$, for $n\in\Nn$. Thus, the sequence $(\phi_n)_{n\in\Nn}$ in $ C_\zeta(\Omega)$ satisfies
\begin{equation}\label{eq: Aux transform f 1}
	\phi_n \leq 0 \quad \mbox{and} \quad \int_{\Omega} \phi_n d\mu \to +\infty.
\end{equation}
Let $\tilde{\phi}_n = \phi_n + vu'_T$, for $n\in\Nn$. By Assumption \ref{hyp: uT C1 Linfty}, we have $v u'_T \in C_\zeta(\Omega)$. Therefore, $\tilde{\phi}_n \in C_\zeta(\Omega)$ for $n\in\Nn$. Moreover, $\int_{\Omega} \tilde{\phi}_n d\mu \to +\infty$ as well.  From \eqref{def: f map} we get
\[
	f(\tilde{\phi}_n) = T \sup_{(t,x,v)\in\Omega} \left( \phi_n - L \right).
\]
By Assumption \ref{hyp: H growth}, using \eqref{eq: L growth out of H} (see Remark \ref{remark:H growth}) and the first condition in \eqref{eq: Aux transform f 1}, we get
\[
	\phi_n(t,x,v) - L(x,v) \leq -C_1|x|^{\gamma_1}-\dfrac{|v|^{\gamma_2'}}{\gamma_2' C^{\gamma_2'/\gamma_2}}+C \leq C.
\]
Thus, $f(\tilde{\phi}_n) \leq  T C$. Hence, we conclude that
\[
	+\infty = \lim_{n} \int_{\Omega} \tilde{\phi}_n d\mu - TC \leq \lim_{n} \int_{\Omega} \tilde{\phi}_n d\mu - f(\tilde{\phi}_n) \leq f^*(\mu). \qedhere
\]
\end{proof}
\begin{proposition}\label{Lemma 5.2 alike}
Let $\zeta$ satisfy \eqref{eq: zetas selection}. Suppose that Assumptions \ref{hyp: H convex}-
\ref{hyp: uT C1 Linfty} hold. Let $\mu \in \mathcal{U}^\zeta$. If $\mu \geq 0$, then 
\[
	f^*(\mu) \geq \int_{\Omega} L + v u'_T ~ d\mu + \sup_{\psi \in C_\zeta(\Omega)} \left( \int_{\Omega} \psi d\mu - T \sup_{\Omega} \psi \right).
\]
\end{proposition}
\begin{proof}
From \eqref{eq: L growth out of H}, we can add a constant $C$ to $L$ and assume that $0\leq L$. Using Remark \ref{remark:H growth}, let $L_n$ be a sequence in $C_\zeta(\Omega)$ such that $0\leq L_n \leq L_{n+1}\leq L$ and $L_n \to L$ pointwise. Fix $n\in \Nn$, $\phi \in C_\zeta(\Omega)$ and let $\psi = \phi - vu'_T -L_n \in C_\zeta(\Omega)$. Then
\begin{align*}
	\int_{\Omega} \phi d\mu - f(\phi) & = \int_{\Omega} \left( \psi + L_n + vu'_T \right) d\mu - T\sup_{\Omega} \left( \psi + L_n - L \right) 
\\
& \geq \int_{\Omega} \left( \psi + L_n + vu'_T \right) d\mu  - T\sup_{\Omega} \psi.
\end{align*}
By the Monotone Convergence Theorem, we have $\int_{\Omega} L_n d\mu \to \int_{\Omega} L d\mu$. Therefore,
\begin{align*}
	f^*(\mu) = \sup_{\phi \in C_\zeta (\Omega)} \int_{\Omega} \phi d\mu - f(\phi) & \geq \int_{\Omega} L + vu'_T ~ d\mu + \sup_{\psi \in C_\zeta (\Omega)} \left( \int_{\Omega} \psi ~ d\mu - T \sup_{\Omega} \psi\right). \qedhere
\end{align*}
\end{proof}
\begin{proposition}\label{Prop: f transform}
Let $\zeta$ satisfy \eqref{eq: zetas selection}. Suppose that Assumptions \ref{hyp: H convex}-
\ref{hyp: uT C1 Linfty} hold. Let $f$ be as in \eqref{def: f map} and let $f^*$ be its Legendre transform; that is,
\[
	f^*(\mu) = \sup_{\phi \in C_\zeta (\Omega)} \left( \int_\Omega \phi ~ d\mu - f(\phi) \right).
\]
Then,
\[
	f^*(\mu)=
	\begin{cases}
	 \int_\Omega L+vu'_T ~ d\mu & \mu \in \mathcal{U}_1
	\\ 
	+\infty & \mbox{otherwise}
	\end{cases}.
\]	
\end{proposition}
\begin{proof}
By Proposition \ref{Lemma5.1. alike}, if $\mu\ngeq 0 $ then $f^*(\mu)=+\infty$.  Let $\mu \geq 0$. If $\mu \not\in \mathcal{U}_1$, by definition, $\int_{\Omega} d\mu \neq T$ (see \eqref{def: U1}). Define $\phi_\alpha = \alpha + vu'_T - C \in C_\zeta(\Omega)$, where $\alpha \in \Rr$ and $C$ is given by Assumption \ref{hyp: H growth}. Then, by \eqref{eq: L growth out of H}, we obtain
\[	
	f(\phi_\alpha) = T\sup_{\Omega} \left( \alpha - C-L\right) \leq T\alpha.
\]
Adding $\alpha \int_{\Omega} d\mu$ and rearranging the previous expression, we get
\[	
	\alpha \int_{\Omega} d\mu- T\alpha \leq \alpha \int_{\Omega} d\mu-f(\phi_\alpha),
\]
which implies that
\[
	 \left( \int_{\Omega} d\mu -T\right) \sup_{\alpha \in \Rr} \alpha \leq \sup_{\alpha \in \Rr} \int_{\Omega} \alpha d\mu - f(\phi_\alpha) \leq f^*(\mu).
\]
From the preceding inequality, we conclude that $f^*(\mu) = +\infty$. On the other hand, if $\mu \in \mathcal{U}_1$, by definition,  $\int_{\Omega} d\mu = T$. For any $\phi \in C_\zeta(\Omega)$, we have
\[
	\int_{\Omega} \phi - L - vu'_T ~ d\mu \leq 	\int_{\Omega} \sup_{\Omega} \left( \phi - L - vu'_T\right) ~ d\mu  = T \sup_{\Omega} \left( \phi - L - vu'_T\right) = f(\phi).
\]
Rearranging the previous inequality, we obtain
\[
	\int_{\Omega} \phi d\mu - f(\phi) \leq \int_{\Omega} L + vu'_T~ d\mu,
\]
and we conclude that $f^*(\mu) \leq \int_{\Omega} L + vu'_T~ d\mu$. 
Finally, we take $\psi \equiv 0$ in Proposition \ref{Lemma 5.2 alike} to obtain $f^*(\mu) \geq \int_{\Omega} L + vu'_T~ d\mu$. The result follows.
\end{proof}
Now, we define the second functional we use in the Fenchel-Rockafellar theorem. Recall the definition of $\mathcal{C}$ in \eqref{def: C set}. Fix (see \eqref{def: U2} and \eqref{def: U3})
\begin{equation}\label{eq: mu bar selection}
\overline{\mu} \in \mathcal{U}_2(B) \cap \mathcal{U}_3.
\end{equation}
Define $g: C_\zeta(\Omega) \to \Rr\cup \{+\infty\}$ as 
\begin{equation}\label{def: g map}
g(\phi)=
\begin{cases} 
\displaystyle -\int_{\Omega} \phi d\overline{\mu}, & \phi \in \mathcal{C} \\ 
+\infty, & \mbox{otherwise}.
 \end{cases}
\end{equation}
\begin{proposition}\label{Prop: g transform}
Let $\zeta$ satisfy \eqref{eq: zetas selection}. Suppose that Assumption \ref{hyp:Q C infty} holds. Assume that $B:C([0,T]\times \Rr) \to \Rr$ is a linear and bounded operator such that $\mathcal{U}_2(B) \cap \mathcal{U}_3 \neq \emptyset$. Then
\[
	g^*(\mu) = 
	\begin{cases} 0, & -\mu \in \mathcal{U}_2(B) \cap \mathcal{U}_3 \\ 
	+\infty, & \mbox{otherwise}.
	\end{cases}
\]
\end{proposition}
\begin{proof}
Let $\mu \in \mathcal{U}^\zeta$ and define $\hat{\mu}=\mu + \overline{\mu} ~ \in \mathcal{R}(\Omega)$. 

Assume that $-\mu \in \mathcal{U}_2(B) \cap \mathcal{U}_3$. Then, $\hat{\mu}$ satisfies 
\[
	\int_{\Omega} A^v\varphi - (v-Q(t))\eta(t) ~ d\hat{\mu}=0
\]
for all $\varphi \in \Lambda ([0,T]\times\Rr)$ and $\eta \in C([0,T])$. Because $\hat{\mu}$ defines a linear and bounded functional on $C_\zeta(\Omega)$ (see Section \ref{remark:Czeta dual}), the continuity under $\|\cdot\|_{\zeta}$ guarantees that $\int_{\Omega} \phi ~d\hat{\mu}=0$ for all $\phi \in \mathcal{C}$; that is, $\int_{\Omega} \phi ~d\mu = -\int_{\Omega} \phi ~ d\overline{\mu}=g(\phi)$ for all $\phi \in \mathcal{C}$. Hence,
\[
	g^*(\mu)=\sup_{\phi \in C_\zeta(\Omega)} \left( \int_{\Omega} \phi~ d\mu - g(\phi) \right) =\sup_{\phi \in \mathcal{C}} \left( \int_{\Omega} \phi~ d\mu - g(\phi) \right) =0.
\]
Now, assume that $-\mu \not\in \mathcal{U}_2(B) \cap \mathcal{U}_3$. Then, either $-\mu \not\in \mathcal{U}_2(B)$ or $-\mu \not\in \mathcal{U}_3$. In the first alternative, there exists $\varphi \in \Lambda([0,T]\times \Rr)$ such that 
\[
	-\int_{\Omega} A^v \varphi d\mu(x,q,s) \neq B \varphi,
\]
we have
\[
	\int_{\Omega} A^v \varphi d\hat{\mu} =\int_{\Omega} A^v \varphi d\mu+\int_{\Omega} A^v \varphi d\overline{\mu} \neq 0.
\]	
Define $\hat{\phi} = A^v \varphi$. Then $\hat{\phi} \in \mathcal{C}$ and 
satisfies $\int_{\Omega} \hat{\phi}~ d\hat{\mu} \neq 0$, and using \eqref{def: g map}, we obtain
\[
	\sup_{\phi \in C_\zeta(\Omega)} \left( \int_{\Omega} \phi ~d\mu - g(\phi) \right)= \sup_{\phi \in \mathcal{C}} \left( \int_{\Omega} \phi ~d\mu - g(\phi) \right) \geq \int_{\Omega} \hat{\phi} ~d\mu - g(\hat{\phi})= \int_{\Omega} \hat{\phi}~ d\hat{\mu}.
\]
Let $\alpha_n = n~\mbox{sgn}\left(\int_{\Omega} \hat{\phi}~ d\hat{\mu}\right)$, where $\mbox{sgn}(\cdot)$ denotes the sign function, and $\hat{\phi}_n = \alpha_n A^v  \varphi$, for $n\in \Nn$. Because $\alpha_n \varphi$ is a sequence in $\Lambda([0,T]\times\Rr)$, $\hat{\phi}_n$ is a sequence in $\mathcal{C}$. Furthermore, the previous inequality implies
\[
	g^*(\mu) \geq n ~\mbox{sgn}\left(\int_{\Omega} \hat{\phi}~ d\hat{\mu}\right) \int_{\Omega} \hat{\phi}~ d\hat{\mu}
\]
for all $n \in \Nn$. Hence $g^*(\mu)=+\infty$.

In the second alternative, there exists $\eta  \in C([0,T])$ such that 
\[
	\int_{\Omega} \eta (t) (v-Q(t)) ~d\mu \neq 0,
\]
we have 
\[
	\int_{\Omega} \eta (v-Q) \varphi d\hat{\mu}=\int_{\Omega} \eta (v-Q) \varphi d\mu \neq 0.
\]
 Define $\hat{\phi} = -(v-Q)\eta$. Then $\hat{\phi}  \in \mathcal{C}$ and
satisfies $\int_{\Omega} \hat{\phi}~ d\hat{\mu} \neq 0$. Proceeding as before, we obtain $g^*(\mu)=+\infty$.
\end{proof}
\begin{theorem}\label{thm: 5.5}
Let $\zeta$ satisfy \eqref{eq: zetas selection}. Suppose that Assumptions \ref{hyp: H convex}-
\ref{hyp:Q C infty} hold. Assume that $B:C([0,T]\times \Rr) \to \Rr$ is a linear and bounded operator such that $\mathcal{U}_2(B) \cap \mathcal{U}_3 \neq \emptyset$. Then
\begin{align*}
	& \inf_{\substack{\varphi , \eta }} \left(  T \sup_{(t,x)} \bigg( -\varphi_t + Q\eta + H(x,\varphi_x + \eta + u'_T)\bigg) -B \varphi \right) 
= \max_{\mu} \left( -\int_{\Omega} L  + vu'_T~ d\mu \right),
\end{align*}
where the supremum is taken over $(t,x) \in [0,T] \times \Rr$, the infimum is taken over $\varphi \in \Lambda([0,T]\times\Rr)$, $\eta \in C([0,T])$, and the maximum is taken on $ \mu \in \mathcal{U}_1\cap \mathcal{U}_2(B) \cap \mathcal{U}_3$.
\end{theorem}
\begin{proof}
Recall that $f$ is convex, and by Lemma \ref{lem: f continuous}, $f$ is continuous on $C_\zeta(\Omega)$. By definition, $g$ is convex. Therefore, to use Theorem \ref{Thm: Fenchel-Rockafellar}, we need to find $\phi \in C_\zeta(\Omega)$ such that $f(\phi),g(\phi)<+\infty$. Take $\varphi(t,x)=Ct - u_T(x)$, where $C$ is given by Assumption \ref{hyp: H growth}. Then $\phi=A^v \varphi = -C + v u'_T \in \mathcal{C}$. By Assumption \ref{hyp: H growth} and \eqref{eq: L growth out of H}, we have
\[
	f(\phi) \leq 0.
\]
From the definition of $g$ (see \eqref{def: g map}),
\[
	g(\phi) = -B\varphi,
\]
and by Assumption \ref{hyp: uT C1 Linfty}, $B\varphi$ is finite. Hence, relying on the duality relation between $C_\zeta(\Omega)$ and $\mathcal{U}^\zeta$ (see Remark \ref{remark:Czeta dual}), we apply Theorem \ref{Thm: Fenchel-Rockafellar} to get
\[
	\inf_{\phi \in C_\zeta(\Omega)} \left( f(\phi)+g(\phi) \right) = \max_{\mu \in \mathcal{U}^\zeta} \left( -f^*(-\mu) - g^*(\mu) \right) = \max_{\mu \in \mathcal{U}^\zeta} \left( -f^*(\mu) - g^*(-\mu) \right).
\]
From Proposition \ref{Prop: f transform} and Proposition \ref{Prop: g transform}, it follows that
\[
	\max_{\mu \in \mathcal{U}^\zeta} \left( -f^*(\mu) - g^*(-\mu) \right) = \max_{\mu \in \mathcal{U}_1 \cap\mathcal{U}_2(B) \cap\mathcal{U}_3 } \left(-\int_{\Omega} L + vu'_T~ d\mu\right).
\]
By \eqref{def: g map},
\[
	\inf_{\phi \in C_\zeta(\Omega)} \left(f(\phi)+g(\phi)\right) 
= \inf_{\phi \in \mathcal{C}} \left( T \sup_{\Omega} \left( \phi - L - vu'_T \right) - \int_{\Omega} \phi d\overline{\mu}\right),
\]
and using the definition of $\mathcal{C}$ in \eqref{def: C set}, the selection of $\overline{\mu}$ in \eqref{eq: mu bar selection}, and the definition of the Legendre transform \eqref{eq:Legendre transform}, we obtain
\begin{align*}
&\inf_{\phi \in \mathcal{C}} \left( T \sup_{\Omega} \left( \phi - L - vu'_T \right) - \int_{\Omega} \phi d\overline{\mu}\right)
	\\
	= & \inf_{\substack{\varphi \in \Lambda([0,T]\times\Rr)\\ \eta \in C([0,T])}} \left( T \sup_{\Omega} \left( A^v \varphi - (v-Q) \eta - L - vu'_T \right) - \int_{\Omega} A^v \varphi - (v-Q) \eta ~d\overline{\mu}\right)
	\\
	= &	\inf_{\substack{\varphi \in \Lambda([0,T]\times\Rr)\\ \eta \in C([0,T])}} \left( T \sup_{\Omega} \left( -\varphi_t + Q \eta - v(\varphi_x + \eta+u'_T)-L \right) - B \varphi \right)
	\\
	= & \inf_{\substack{\varphi \in \Lambda([0,T]\times\Rr)\\ \eta \in C([0,T])}} \left( T \sup_{(t,x) \in [0,T] \times \Rr} \left( -\varphi_t + Q \eta +\sup_{v\in\Rr} \left(- v(\varphi_x + \eta+u'_T)-L\right) \right) - B \varphi \right)
	\\
	= &	\inf_{\substack{\varphi \in \Lambda([0,T]\times\Rr)\\ \eta \in C([0,T])}}  \left( T \sup_{(t,x)\in [0,T] \times \Rr} \bigg( -\varphi_t + Q \eta +H(x,\varphi_x + \eta+u'_T) \bigg) - B \varphi \right).
\end{align*}
The result follows.
\end{proof}
Now, we use the duality result from Theorem \ref{thm: 5.5} to prove Theorem \ref{thm:convex duality h}.
\begin{proof}[\bf{Proof of Theorem \ref{thm:convex duality h}}]
Let $B$ be given by \eqref{def: B operator} and consider the set $\mathcal{U}_2(B)$ according to \eqref{def: U2}. By assumption, $\mathcal{H}(m_0,\nu_T) \neq \emptyset$, from which \eqref{def: h} and Remark \ref{remark: H equals Us intersection} imply 
\[
h(m_0, \nu_T) = \inf_{\substack{\mu \in \mathcal{H}(m_0,\nu_T)}}  \int_{\Omega} L + v u_T'd\mu =\inf_{\mu \in \mathcal{U}_1 \cap\mathcal{U}_2(B) \cap\mathcal{U}_3 } \int_{\Omega} L + vu'_T~ d\mu.
\]
In particular, $\mathcal{U}_2(B) \cap \mathcal{U}_3 \neq \emptyset$. The conclusion follows by invoking Theorem \ref{thm: 5.5} and the previous equality.
\end{proof}


\section{Preliminary results on MFG}\label{Sec:MFG Results}
Here, we consider approximations of Lipschitz continuous solutions of the Hamilton-Jacobi equation in \eqref{eq:MFG system}. We provide a commutation lemma, which states that the approximated solutions are sub-solutions of an approximate Hamilton-Jacobi equation. Then, we improve the result in \cite{gomes2018mean}, where the authors proved that $\varpi$ solving \eqref{eq:MFG system} and \eqref{eq:initial terminal} satisfies $\varpi \in W^{1,1}([0,T])$. A better result can be established as $\varpi$ is Lipschitz continuous, as we prove here. This result, in turn, enables the use of the commutation lemma.

\subsection{A commutation lemma}
The commutation lemmas presented in \cite{gomes2020large} and \cite{MTL17} are applied to a Hamilton-Jacobi equation where the state variable is constrained to the $d$-dimensional torus; that is, periodic boundary conditions. Here, we present a version of this lemma that is valid for the non-periodic case and takes into account the dependence of the Hamilton-Jacobi equation on the price variable. 

We start by introducing smooth approximations to the solutions of \eqref{eq:MFG system}. Let $\rho, \theta \in C^\infty_c(\Rr; [0, \infty))$ be symmetric standard mollifiers, i.e. 
\[
\supp\rho, \supp\theta \subset [-1,1], ~\rho(t) = \rho(-t), ~ \theta(x) = \theta(-x), ~ \mbox{and}~ \| \rho \|_{L^1(\Rr)} = \| \theta \|_{L^1(\Rr)}= 1.
\]  
For $ 0 < \alpha <  T$,  set  $\rho^\alpha(t) := \alpha^{-1} \rho(\alpha^{-1} t),\, t \in \Rr$ and $\theta^\alpha(x) := \alpha^{-1} \theta(\alpha^{-1} x),\, x \in \Rr$.
Then, we have that $\| \rho ^\alpha\|_{L^1(\Rr)} = \| \theta^\alpha \|_{L^1(\Rr)}= 1$, and
\begin{equation}\label{rho-theta-bounds}
\int_0^\infty \rho^\alpha(s) \int_{\Rr} \theta^\alpha(y) |y|~dy ds ,~ \int_0^\infty \rho^\alpha(s)  \int_{\Rr} \theta^\alpha(y) s~dy ds \leq \alpha.
\end{equation}
For $w \in C([\alpha, T] \times \Rr)$, define $w^\alpha \in C^\infty([\alpha, T] \times \Rr)$ as
\begin{equation}\label{eq:convolution}
	w^\alpha(t, x) = \int_0^\infty \rho^\alpha(s) \int_{\Rr} \theta^\alpha(y) w(t-s,x-y) dy ds, \quad (t, x) \in [\alpha, T] \times \Rr.
	\end{equation}
	
\begin{lemma}\label{lem: subsolution}
Suppose that Assumptions \ref{hyp: H convex} and \ref{hyp: H growth} hold. Let $(w,m,\varpi)$ solve \eqref{eq:MFG system} (see Remark \ref{Remark: Notion Sol MFG}). Assume further that $w$ is Lipschitz in $x$ and $\varpi$ is Lipschitz. Let $w^\alpha$ be defined as in \eqref{eq:convolution}.	
	Then, there exists $C'>0$ depending on $\varpi$, $H$ and the Lipschitz constants of  $w$ and $\varpi$ such that 
	\begin{equation}\label{eq:w subsolution}
	-w^\alpha_t + H(x, \varpi + w^\alpha_x) \leq C' \alpha, \quad \mbox{for all}\quad  (t,x)\in [\alpha, T] \times \Rr.
	\end{equation}
\end{lemma}

\begin{proof}
	To obtain the desired inequality, we write the left-hand side of \eqref{eq:w subsolution} as a convolution between $\rho^\alpha \theta^\alpha$ and the left-hand side of the first equation in \eqref{eq:MFG system}. Thus, for the first term, we have
	\begin{equation}\label{eq:wt-alpha}
		-w^\alpha_t(t, x) = \int_0^\infty \rho^\alpha(s) \int_{\Rr} \theta^\alpha(y) (-w_t(t-s,x-y)) dy ds.
	\end{equation}

	For the second term, by Jensen's inequality (\cite{HardyInequalities}, Theorem 204), we have 
	\begin{align}\label{eq: Aux Subsol lemma 0}
	H(x, \varpi(t) + w^\alpha_x(t, x)) &= H\left(x, \int_0^\infty \rho^\alpha(s) \int_{\Rr} \theta^\alpha(y) (\varpi(t) + w_x(t-s,x-y)) dy ds \right) \nonumber \\
	&\leq  \int_0^\infty \rho^\alpha(s) \int_{\Rr} \theta^\alpha(y) H(x, \varpi(t) + w_x(t-s, x-y)) dy ds.
	\end{align}
	Let $t \in [\alpha, T]$, $s \in [0, \alpha]$, $x,y\in\Rr$, and 
\[
	q(t,x;s,y):=H\big( x, \varpi(t) + w_x(t-s, x-y)\big)-H\big(x-y, \varpi(t-s) + w_x(t-s, x-y)\big).
\]
	Using Assumption \ref{hyp: H growth} and the Lipschitz continuity of $w$ and $\varpi$, we get	
	\begin{align*}
	|q(t,x;s,y)| &\leq \left|H\big(x, \varpi(t) + w_x(t-s, x-y)) - H(x, \varpi(t-s) + w_x(t-s, x-y)\big)\right|
	\\&+\left|H\big(x, \varpi(t-s) + w_x(t-s, x-y)) - H(x-y, \varpi(t-s) + w_x(t-s, x-y)\big)\right|
	\\ & \leq C|\varpi(t)-\varpi(t-s)|\left(\left| \varpi(t-s) + w_x(t-s, x-y)\right|^{\gamma_2-1}+1\right) 
	\\ &+ C|y|\big( \left|\varpi(t-s)+w_x(t-s,x-y)\right|^{\gamma_2}+1\big)
	\\ & \leq C' s \left(\left| \varpi(t-s) + w_x(t-s, x-y)\right|^{\gamma_2-1}+1\right) 
	\\ &+ C|y|\big( \left|\varpi(t-s)+w_x(t-s,x-y)\right|^{\gamma_2}+1\big)
	\\ &\leq C'( s+|y|),	
	\end{align*}
	where $C'$ depends on $\varpi$, $\gamma_2$, and the Lipschitz constants of $w$ and $\varpi$. From the previous inequality and \eqref{rho-theta-bounds}, we obtain
	\begin{align}\label{eq: Aux Subsol lemma 1}
	\int_0^\infty \rho^\alpha(s) \int_{\Rr} \theta^\alpha(y) q(t,x;s,y) dy ds & \leq \int_0^\infty \rho^\alpha(s) \int_{\Rr} \theta^\alpha(y) C'( s+|y|) dy ds \leq  C'\alpha.
	\end{align}
Then, from \eqref{eq: Aux Subsol lemma 0} and \eqref{eq: Aux Subsol lemma 1}, we have
\begin{align*}
	&H(x, \varpi(t) + w^\alpha_x(t, x))  
	\\
	& \leq  \int_0^\infty \rho^\alpha(s) \int_{\Rr} \theta^\alpha(y) H(x-y, \varpi(t-s) + w_x(t-s, x-y)) dy ds + C' \alpha.
\end{align*}
Using the preceding inequality and \eqref{eq:wt-alpha}, we get
\begin{align*}
&	-w^\alpha_t + H(x, \varpi + w^\alpha_x) 
\\
&\leq \int_0^\infty \rho^\alpha(s) \int_{\Rr} \theta^\alpha(y) \big( -w_t(t-s,x-y) + H(x-y, \varpi(t-s) + w_x(t-s, x-y)) \big) dyds 
\\
& \quad + C' \alpha,
\end{align*}
which implies \eqref{eq:w subsolution}.
\end{proof}

\subsection{Lipschitz continuity of the price}
We begin by recalling the following techniques and results from \cite{gomes2018mean} if Assumptions \ref{hyp:Q C infty}, \ref{hyp: H separability DS}, and \ref{hyp: V-uT Lipschitz 2nd D bounded DS} hold. Firstly, to prove the existence of a solution $(u,m,\varpi)$ of \eqref{eq:MFG system} and \eqref{eq:initial terminal}, the authors used the vanishing viscosity method, which relies on the following regularized version of \eqref{eq:MFG system}
\begin{equation}\label{eq:regularized MFG}
	\begin{cases}
	-u_t(t,x) + H(x,\varpi(t) + u_x(t,x) ) = \epsilon u_{xx}
	\\
	m_t(t,x) - \big(H_p(x,\varpi(t) + u_x(t,x))m(t,x)\big)_x = \epsilon m_{xx}(t,x)
	\\
	-\int_{\Rr} H_p(x,\varpi(t) + u_x(t,x))m(t,x) dx = Q(t)
	\end{cases} \quad (t,x)\in [0,T]\times \Rr,
\end{equation}
subject to \eqref{eq:initial terminal}, where $\epsilon > 0$. Secondly, the proof of existence of a solution  $(u^\epsilon,m^\epsilon,\varpi^\epsilon)$ of \eqref{eq:regularized MFG} and \eqref{eq:initial terminal} uses a fixed-point argument. This argument shows that $(u^\epsilon,m^\epsilon,\varpi^\epsilon)$ satisfies
\begin{equation}\label{eq:FixPoint argument}
	\dot{\varpi^\epsilon}=\frac{-\dot{Q}-\int_{\Rr}H_{pp}(x,\varpi^\epsilon+u^\epsilon_x)H_x(x,\varpi^\epsilon+u^\epsilon_x)m^\epsilon + \epsilon H_{ppp}(x,\varpi^\epsilon+u^\epsilon_x)(u^\epsilon_{xx})^2m^\epsilon~ dx}{\int_{\Rr} H_{pp}(x,\varpi^\epsilon+u^\epsilon_x) m^\epsilon ~ dx},
\end{equation}
and $\varpi^\epsilon(0)$ is determined by
\[
	\int_{\Rr}H_p(x,\varpi^\epsilon(0)+u^\epsilon(0,x))m_0(x)~dx = -Q(0).
\]
Using \eqref{eq:FixPoint argument}, we can deduce the Lipschitz continuity of $\varpi$, where $(u,m,\varpi)$ solves \eqref{eq:MFG system} and \eqref{eq:initial terminal}, as we show next.
\begin{proposition}\label{prop:price Lipschitz}
Suppose that Assumptions \ref{hyp: H convex}, \ref{hyp:Q C infty}, \ref{hyp: H separability DS} and \ref{hyp: V-uT Lipschitz 2nd D bounded DS} hold. Then, there exists a solution $(u,m,\varpi)$ of \eqref{eq:MFG system} and \eqref{eq:initial terminal} such that $\varpi$ is Lipschitz continuous.
\end{proposition}
\begin{proof}
The existence of a solution $(u,m,\varpi)$ of \eqref{eq:MFG system} and \eqref{eq:initial terminal} is guaranteed by Theorem 1 in \cite{gomes2018mean}. We aim to prove that $\varpi$, obtained in \cite{gomes2018mean}, is Lipschitz. To obtain this solution, the authors considered, for $\epsilon>0$, solutions $(u^\epsilon,m^\epsilon,\varpi^\epsilon)$ of \eqref{eq:regularized MFG} and \eqref{eq:initial terminal}  that satisfy \eqref{eq:FixPoint argument}. Extracting a sub-sequence if necessary, it is guaranteed that $\varpi^\epsilon \to \varpi$ uniformly. To prove that $\varpi$ is Lipschitz, we consider the right-hand side of \eqref{eq:FixPoint argument}. By Assumption \ref{hyp: H convex}, we have
\begin{align}\label{eq:Aux price Lip 1}
	t\mapsto \frac{1}{\int_{\Rr} H_{pp}(x,\varpi^\epsilon+u^\epsilon_x) m^\epsilon ~ dx} \leq \frac{1}{\kappa} \quad \mbox{for all} \quad t\in[0,T].
\end{align}
By Assumptions \ref{hyp: H separability DS} and \ref{hyp: V-uT Lipschitz 2nd D bounded DS}, $|H_x|=|V'|\leq \mbox{Lip}(V)$, where $\mbox{Lip}(V)$ denotes the Lipschitz constant of $V$. Hence, Assumption \ref{hyp: H separability DS} 
implies that 
\begin{align}\label{eq:Aux price Lip 2}
	t\mapsto \int_{\Rr}H_{pp}(x,\varpi^\epsilon+u^\epsilon_x)H_x(s,\varpi^\epsilon+u^\epsilon_x)m^\epsilon~ dx \leq \tfrac{\mbox{Lip}(V)}{\kappa'} \quad \mbox{for all} \quad t\in[0,T].
\end{align}
By Assumption \ref{hyp: H separability DS} and Assumption \ref{hyp: H convex}, we have
\begin{align}\label{eq:Aux price Lip 3}
	\int_0^T \int_{\Rr} |H_{ppp}(x,\varpi^\epsilon+u_x^\epsilon)|(u_{xx}^\epsilon)^2m^\epsilon dx dt & \leq C \int_0^T \int_{\Rr} (u_{xx}^\epsilon)^2 m^\epsilon dx dt \nonumber
	\\
	& \leq \frac{C}{\kappa} \int_0^T \int_{\Rr} H_{pp}(x,\varpi^\epsilon+u_x^\epsilon)(u_{xx}^\epsilon)^2 m^\epsilon dx dt.
\end{align}
Assumptions \ref{hyp: H separability DS}, \ref{hyp: V-uT Lipschitz 2nd D bounded DS} and Proposition 5 in \cite{gomes2018mean} guarantee that the term 
\begin{equation}\label{second-eng-est}
	\int_0^T \int_{\Rr} H_{pp}(x,\varpi^\epsilon+u_x^\epsilon)(u_{xx}^\epsilon)^2 m^\epsilon \dx \dt
\end{equation}
has an upper bound that is independent of $\epsilon$. Hence, using Assumption \ref{hyp:Q C infty}, \eqref{eq:Aux price Lip 1}, \eqref{eq:Aux price Lip 2} and \eqref{eq:Aux price Lip 3}, we can write \eqref{eq:FixPoint argument} as
\begin{equation*}
	\dot{\varpi^\epsilon}=\vartheta^\epsilon_\infty +\epsilon \vartheta_1^\epsilon,
\end{equation*}
where 
\begin{align*}
	\vartheta_{\infty}^\epsilon = \frac{-\dot{Q}-\int_{\Rr}H_{pp}(x,\varpi^\epsilon+u^\epsilon_x)H_x(x,\varpi^\epsilon+u^\epsilon_x)m^\epsilon dx}{\int_{\Rr} H_{pp}(x,\varpi^\epsilon+u^\epsilon_x) m^\epsilon ~ dx} \in L^\infty([0,T]),
\end{align*}
\begin{align*}
	\vartheta_{1}^\epsilon = \frac{-\int_{\Rr} H_{ppp}(x,\varpi^\epsilon+u^\epsilon_x)(u^\epsilon_{xx})^2m^\epsilon~ dx}{\int_{\Rr} H_{pp}(x,\varpi^\epsilon+u^\epsilon_x) m^\epsilon ~ dx} \in L^1([0,T]),
\end{align*}
and they satisfy
\[
	\|\vartheta_1^\epsilon\|_{L^1([0,T])} \leq C' \quad \mbox{and} \quad \|\vartheta_\infty^\epsilon\|_{L^\infty([0,T])} \leq C'
\]
for $\epsilon \to 0$, where $C'$ is independent of $\epsilon$. Hence, (\cite{FoLe07}, Proposition 1.202) passing to a sub-sequence, there exists $\mu \in \mathcal{R}([0,T])$ such that $\vartheta^\epsilon_1$ converges in the weak-$\star$ topology to $\mu$; that is,
\begin{equation}\label{eq:Aux weak lim1}
	\int_0^T \vartheta^\epsilon_1 \eta ~dt \to \int_0^T \eta ~d\mu \quad \mbox{for all}\quad \eta \in C([0,T]).
\end{equation}
Passing to a further sub-sequence if necessary, (\cite{FoLe07}, Proposition 2.46) there exists $\vartheta_{\infty} \in L^{\infty}([0,T])$ such that $\vartheta_{\infty}^\epsilon$ converges in the weak-$\star$ topology to $\vartheta_\infty$; that is,
\begin{equation}\label{eq:Aux weak lim2}
	\int_0^T \vartheta^\epsilon_\infty \eta ~dt \to \int_0^T \vartheta_\infty \eta ~dt \quad \mbox{for all}\quad \eta \in L^1([0,T]).
\end{equation}
 Let $\eta \in C^1_c((0,T))$. By uniform convergence, we have that 
 \[
 	\int_0^T \left(\vartheta^\epsilon_\infty +\epsilon \vartheta_1^\epsilon\right) \eta ~ dt=\int_0^T \dot{\varpi^\epsilon} \eta ~ dt=-\int_0^T \varpi^\epsilon \dot{\eta} ~ dt \to -\int_0^T \varpi \dot{\eta} ~dt ,
 \]
and by \eqref{eq:Aux weak lim1} and \eqref{eq:Aux weak lim2}, we have that
  \[
 	\int_0^T \left(\vartheta^\epsilon_\infty +\epsilon \vartheta_1^\epsilon\right) \eta ~ dt \to \int_0^T \vartheta_{\infty} \eta~ dt.
 \]
 Hence, $\dot{\varpi} = \vartheta_{\infty}$ in the sense of distributions. Thus, $\varpi \in W^{1,\infty}([0,T])$, which is equivalent to (\cite{evansgariepy2015}, Theorem 4.5) 
$\varpi$ being Lipschitz continuous in $[0,T]$.
\end{proof}

\section{Proof of Theorem \ref{thm: Representation}}
\label{Sec:Proof Theorem 1.2}
Here, we use the results from Sections \ref{Sec:Duality} and \ref{Sec:MFG Results} to prove Theorem \ref{thm: Representation}. We divide the proof into two lemmas, Lemma \ref{lem: lower bound} and Lemma \ref{lem: upper bound}.
\begin{lemma}
\label{lem: lower bound}
Let $m_0 \in \Pp(\Rr)$. Suppose that Assumptions \ref{hyp: H convex}-
\ref{hyp: V-uT convex DS} hold. Let $(u,m,\varpi)$ solve \eqref{eq:MFG system} and \eqref{eq:initial terminal}. Then, 
\begin{equation*}
	\int_{\Rr} \left( u(0, x) - u_T(x) \right)dm_0(x) - \int_0^T \varpi(t) Q(t) dt \leq 
	\inf_{\substack{\mu \in \mathcal{H}(m_0)}} \int_{\Omega} L(x,v) + v u_T'(x)d\mu(t,x,v).
\end{equation*}
\end{lemma}
\begin{proof}
By Assumptions \ref{hyp:Q C infty}, \ref{hyp: H separability DS}, \ref{hyp: V-uT Lipschitz 2nd D bounded DS}, and \ref{hyp: V-uT convex DS}, Theorem 1 in \cite{gomes2018mean} guarantees the existence of a unique $(u,m,\varpi)$ solving \eqref{eq:MFG system} and \eqref{eq:initial terminal}. Because $u$ is continuous (see Remark \ref{Remark: Notion Sol MFG}), let $w^\alpha$ be the function given by \eqref{eq:convolution}; that is,
\begin{equation}\label{eq:w subsolution u}
w^\alpha(t, x) = \int_0^\infty \rho^\alpha(s) \int_{\Rr} \theta^\alpha(y) u(t-s,x-y) dy ds, \quad (t, x) \in [\alpha, T] \times \Rr.
\end{equation}
For $(t, x) \in [0, T] \times \Rr $, set
\begin{equation*}
u^\alpha(t, x) = 
w^\alpha \left( \tfrac{T-\alpha} {T} t + \alpha, x \right) - u_T(x),
\end{equation*}
which is $C^{1}([0,T]\times \Rr)$ due to Assumption \ref{hyp: uT C1 Linfty}  and \eqref{eq:w subsolution u}. By Assumptions \ref{hyp: H separability DS} and \ref{hyp: V-uT Lipschitz 2nd D bounded DS}, the map $x\mapsto u(t,x)$ is Lipschitz for $0\leq t \leq T$ (\cite{gomes2018mean}, Proposition 1), and the Lipschitz constant depends on $T$ and the estimates for $V$ and $u_T$. Hence, $u_x$ is bounded independently of $t$. Therefore, $u_x^\alpha \in L^\infty([0,T]\times \Rr)$ because
\begin{align*}
	u_x^\alpha(t,x)&=w^\alpha_x\left(\tfrac{T-\alpha}{T}t+\alpha,x\right)-u_T'(x) \nonumber
	\\
	&=\int_0^\infty\rho^\alpha(s)\int_\Rr \theta^\alpha(y)u_x\left(\tfrac{T-\alpha}{T}t+\alpha-s,x-y\right)dyds-u_T'(x).
\end{align*}
Furthermore, recalling that $u$ is a viscosity solution to the first equation in \eqref{eq:regularized MFG} with $\varepsilon=0$, we have that the first equation in \eqref{eq:regularized MFG} with $\varepsilon=0$ holds a.e. $(t,x)\in (0,T)\times\Rr$. Using this and the facts that $u_x \in L^\infty((0,T)\times \Rr)$ and $\varpi \in W^{1,\infty}([0,T])$, we deduce that $u_t\in L_{loc}^\infty((0,T)\times \Rr)$. Thus,  $u_t^\alpha\in L^\infty((0,T)\times \Rr)$ because
\begin{align*}
	u_t^\alpha(t,x)&=\tfrac{T-\alpha}{T}w^\alpha_t\left(\tfrac{T-\alpha}{T}t+\alpha,x\right) \nonumber
	\\
	&=\tfrac{T-\alpha}{T}\int_0^\infty\rho^\alpha(s)\int_\Rr \theta^\alpha(y)u_t\left(\tfrac{T-\alpha}{T}t+\alpha-s,x-y\right)dyds. \nonumber
\end{align*}
Hence, $u^\alpha \in \Lambda([0,T]\times\Rr)$. 
Now, take $\mu \in \mathcal{H}(m_0)$ (see Remark \ref{Remark:H(m0) not empty}). By \eqref{eq:holonomy equality}, we have
\begin{equation}\label{eq:Aux 0 Lower bound Thm 1.2}
\int_{\Omega} u^\alpha_t (t, x) + v u^\alpha_x(t, x) d\mu (t,x,v) = \int_{\Rr} u^\alpha (T, x) d\nu^{\mu}(x) - \int_{\Rr} u^\alpha (0, x) dm_0(x).
\end{equation}
By Assumption \ref{hyp: H convex} and \eqref{eq:Legendre transform}, using $p=\varpi\left(\tfrac{T-\alpha}{T}t+\alpha\right) + u^\alpha_x(t, x) + u_T'(x)$, it follows that 
\begin{align}\label{eq:Aux 1 Lower bound Thm 1.2}
-v u^\alpha_x(t, x) & \leq  L(x,v) + v u_T'(x) + v  \varpi\left(\tfrac{T-\alpha}{T}t+\alpha\right) \nonumber
\\
&\quad + H\big(x,\varpi\left(\tfrac{T-\alpha}{T}t+\alpha\right) + u^\alpha_x(t, x) + u_T'(x)\big).
\end{align}
Moreover, by the Lipschitz continuity of $u$ in $x$ and Proposition \ref{prop:price Lipschitz}, we apply Lemma \ref{lem: subsolution} to $w$ defined by \eqref{eq:w subsolution u} to get  
\[
	-w^\alpha_t\left(\tfrac{T-\alpha}{T}t+\alpha,x\right) + H\left(x, \varpi\left(\tfrac{T-\alpha}{T}t+\alpha\right) + w^\alpha_x\left(\tfrac{T-\alpha}{T}t+\alpha,x\right)\right) \leq C' \alpha, 
\]
for $(t,x)\in [0, T] \times \Rr$; that is,
\begin{equation}\label{eq:Aux 2 Lower bound Thm 1.2}
- \tfrac{T}{T-\alpha}u^\alpha_t +H\left(x, \varpi\left(\tfrac{T-\alpha}{T}t+\alpha\right) + u^\alpha_x\left(t,x\right)+u_T'(x)\right) \leq C' \alpha, \quad \text{ for } (t,x)\in [0, T] \times \Rr.
\end{equation}
Therefore, by \eqref{eq:Aux 0 Lower bound Thm 1.2}, \eqref{eq:Aux 1 Lower bound Thm 1.2}, \eqref{eq:Aux 2 Lower bound Thm 1.2}, and using $\varphi=t$ in \eqref{eq:holonomy converge}, we have
\begin{align*}
	& \int_{\Rr} u^\alpha (0, x) dm_0(x) - \int_{\Rr} u^\alpha (T, x) d\nu^\mu(x) 
	\\
	& = \int_{\Omega} - u^\alpha_t(t, x) - v u^\alpha_x(t, x) d\mu(t,x,v) 
	\\
	& \leq  \int_{\Omega}  - u^\alpha_t d\mu (t,x,v)   +\int_{\Omega} L(x,v) + v u_T' d\mu(t,x,v) 
	\\
	&\quad + \int_{\Omega}  v \varpi\left(\tfrac{T-\alpha}{T}t+\alpha\right) + H\left(x, \varpi\left(\tfrac{T-\alpha}{T}t+\alpha\right) + u^\alpha_x+u_T'\right) d\mu (t,x,v) 
	\\
	&\leq  \int_{\Omega} L(x,v) + v u_T' d\mu(t,x,v) 	+ \int_{\Omega}  v \varpi\left(\tfrac{T-\alpha}{T}t+\alpha\right) d\mu (t,x,v) +(T-\alpha)C'\alpha 
	\\
	&\quad +\int_{\Omega} H\left(x, \varpi\left(\tfrac{T-\alpha}{T}t+\alpha\right) + u^\alpha_x+u_T'\right)  -\tfrac{T-\alpha}{T}H\left(x, \varpi\left(\tfrac{T-\alpha}{T}t+\alpha\right) + u^\alpha_x+u_T'\right)	d\mu (t,x,v).
\end{align*}
Taking $\alpha \to 0$ in the previous inequality and using \eqref{eq:holonomy balance} with $\eta = \varpi$, we obtain
\begin{equation}\label{eq:Aux 3 Lower bound Thm 1.2}
	\int_{\Rr} \left(u(0, x)-u_T(x)\right) dm_0(x) \leq \int_{\Omega} L(x,v) + v u_T' d\mu(t,x,v) 	+\int_{\Omega}  Q(t) \varpi(t) d\mu (t,x,v).
\end{equation}
Finally, taking $\varphi(t,x)=\int_0^t Q(s)\varpi(s)ds$ in \eqref{eq:holonomy equality}, we have
\[
	\int_{\Omega} Q(t)\varpi(t) d\mu(t,x,v) = \int_0^T Q(t) \varpi(t) dt.
\]
Hence, \eqref{eq:Aux 3 Lower bound Thm 1.2} becomes
\begin{equation*}
	\int_{\Rr} \left(u(0, x)-u_T(x)\right) dm_0(x) -\int_0^T Q(t) \varpi(t) dt \leq \int_{\Omega} L(x,v) + v u_T' d\mu(t,x,v).
\end{equation*}
Since  $\mu \in \mathcal{H}(m_0)$ is arbitrary, the preceding inequality completes the proof.
\end{proof}

For the second part of the proof of Theorem \ref{thm: Representation}, we rely on \eqref{eq:regularized MFG}, the regularized version of \eqref{eq:MFG system},
subject to \eqref{eq:initial terminal}. We recall that, if Assumptions \ref{hyp:Q C infty}, \ref{hyp: H separability DS} and \ref{hyp: V-uT Lipschitz 2nd D bounded DS}  hold, (\cite{gomes2018mean}, Theorem 1) there exists a solution $(u^\epsilon, m^\epsilon, \varpi^\epsilon)$ of \eqref{eq:regularized MFG} and \eqref{eq:initial terminal}, where $u^\epsilon$ is a viscosity solution of the first equation, Lipschitz and semiconcave in $x$, and differentiable $m^\epsilon$-almost everywhere, $m^\epsilon \in C([0, T ], \Pp(\Rr) )$ w.r.t. the $1$-Wasserstein distance, and $\varpi^\epsilon \in  W^{1,1}([0,T])$ is continuous.
Moreover, if $\epsilon > 0$ or $\epsilon=0$ and Assumption \ref{hyp: V-uT convex DS} holds, this solution is unique. Using the previous results for the solution of \eqref{eq:regularized MFG} when $\epsilon>0$, we take $\epsilon \to 0$ to exhibit a measure $\mu \in \mathcal{H}(m_0)$ for which the inequality 
\begin{equation*}
	\int_{\Rr} u(0, x) - u_T(x) dm_0(x) - \int_0^T \varpi(t) Q(t) dt \geq 
	 \int_0^T \int_{\Rr^2} L(x,v) + v u_T'(x)d\mu(t,x,v)
\end{equation*}
holds. We begin by establishing the following moment estimate for the probability measures $m^\epsilon$ when $\epsilon>0$.
\begin{proposition} \label{prop: gamma1 moment mu}
	Suppose Assumptions \ref{hyp: H convex}, \ref{hyp: H growth}, \ref{hyp:Q C infty}, \ref{hyp: H separability DS}-
	\ref{hyp: V-uT convex DS} hold. Assume further that $m_0$ satisfies Assumption \ref{hyp: m_0 moment}. Then, for $0<\epsilon<1$, the solution $(u^\epsilon,m^\epsilon,\varpi^\epsilon)$ of \eqref{eq:regularized MFG}-\eqref{eq:initial terminal} satisfies
	\begin{equation*}
		\int_{\Rr}|x|^{\gamma} m^\epsilon(t,x)dx < C \quad \mbox{for almost every}~  t \in [0,T],
	\end{equation*} 
	where the constant $C$ is independent of $\epsilon$.
\end{proposition} 
\begin{proof}
	By Assumptions \ref{hyp: H separability DS}, \ref{hyp: V-uT Lipschitz 2nd D bounded DS} and \ref{hyp: V-uT convex DS}, there exist a unique solution $(u^\epsilon,m^\epsilon,\varpi^\epsilon)$ of \eqref{eq:regularized MFG} and \eqref{eq:initial terminal} (\cite{gomes2018mean}, Theorem 1). Then, by Assumptions \ref{hyp: H convex}, \ref{hyp: H growth}, \ref{hyp:Q C infty}, \ref{hyp: H separability DS} and \ref{hyp: V-uT Lipschitz 2nd D bounded DS}, and the bounds on $\varpi^\epsilon$ and $u_x^\epsilon$ (see Proposition \ref{prop:price Lipschitz}  and \cite{gomes2018mean}, Propositions 1 and 6), we have, for $0<\epsilon<1$,
	\begin{align}\label{eq:aux Hp bound}
		|H_p(\varpi^\epsilon(t)+u^\epsilon_x(t,x))| & \leq C\left(|\varpi^\epsilon(t)+u^\epsilon_x(t,x)|^{\gamma_2-1}+1\right)\nonumber
		\\
		& \leq C\left(C'(\gamma_2)\left(\|\varpi^\epsilon\|_{\infty}^{\gamma_2-1} + \mbox{Lip}(u^\epsilon)^{\gamma_2-1} \right) +1 \right) \nonumber
		\\
		& \leq C\left(C'(\gamma_2)\left(\epsilon^{\gamma_2-1} C'+ \mbox{Lip}(u^\epsilon)^{\gamma_2-1} \right) +1 \right) \nonumber
		\\
		& =C_1 \epsilon^{\gamma_2-1} + C_2 \nonumber
		\\
		& \leq \tilde{C},
	\end{align}
	where $C'(\gamma_2)=\max\{2^{\gamma_2-2},1\}$ and $\mbox{Lip}(u^\epsilon)$, and therefore $C_1$, $C_2$, and $\tilde{C}$ are independent of $\varpi^\epsilon$ and $\epsilon$. Furthermore, $u^\epsilon$ defines the optimal feedback in a stochastic optimal control problem, for which the optimal trajectory satisfies 
	\begin{equation}\label{eq:aux stochastic dynamics}
		\mathrm{d}\mathrm{x}_t = -H_p(\mathrm{x}_t,\varpi^\epsilon(t) + u_x^\epsilon(t,\mathrm{x}_t))\mathrm{d}t + \sqrt{2 \epsilon}\mathrm{d}W_t,
	\end{equation}
	where $W_t$ is a one-dimensional Brownian motion (see \cite{gomes2018mean}). Using Assumptions \ref{hyp: H growth} and \ref{hyp: H separability DS}, the vector field
	\[
		(t,x) \mapsto H_p(x,\varpi^\epsilon(t) + u_x^\epsilon(t,x))=H_p(\varpi^\epsilon(t) + u_x^\epsilon(t,x))
	\]
is bounded and  uniformly Lipschitz. Hence, $m(t,\cdot)=\mathcal{L}(\mathrm{x}_t)$, where $\mathcal{L}(\mathrm{x})$ denotes the law of the random variable $\mathrm{x}$, is a weak solution of the second equation in \eqref{eq:regularized MFG} (\cite{cardaliaguet2018short}, Lemma 4.2.3), and by Assumption \ref{hyp: V-uT convex DS}, this weak solution is unique.  Hence $m^\epsilon(t,\cdot)=\mathcal{L}(\mathrm{x}_t)$. Writing \eqref{eq:aux stochastic dynamics} as
	\[
	\mathrm{x}_t = x + \int_0^t -H_p(\mathrm{x}_t,\varpi^\epsilon(t) + u_x^\epsilon(t,\mathrm{x}_t))\mathrm{d}t +  \int_0^t  \sqrt{2 \epsilon}\mathrm{d}W_t,
	\]
	where $x \in \Rr$, and using \eqref{eq:aux Hp bound}, we have, for $0<\epsilon<1$,
	\begin{align}\label{eq:aux trajectory zeta1 moment}
		|\mathrm{x}_t|^{\gamma} & \leq  2^{^{\gamma}-1} \left( |x|^{\gamma} + 2^{\gamma-1}\left( T^{\gamma}\tilde{C}^{\gamma} +  \sqrt{2 \epsilon}^{\gamma} |W_t|^{\gamma}\right)\right) \nonumber
		\\
		& \leq 2^{^{\gamma}-1} |x|^{\gamma} + C_1' + C_2' |W_t|^{\gamma}, 
	\end{align}
	where $C_1'$ and $C_2'$ are independent of $\varpi^\epsilon$ and $\epsilon$. Because $W_t$ is normally distributed w.r.t. the measure $m^\epsilon(t,x)dx$ in $\Rr$, we have 
	\begin{equation*}
		\Ee[|W_t|^{\gamma}]=\int_{\Rr} |W_t|^{\gamma} m^\epsilon(t,x)dx = \frac{(2t)^{\gamma/2}}{\pi}\Gamma(\tfrac{\gamma+1}{2}),
	\end{equation*}
	where $\Gamma(\cdot)$ denotes the Gamma function. Integrating \eqref{eq:aux trajectory zeta1 moment} w.r.t. $m^\epsilon(t,x)dx$, using the previous formula, and recalling the initial condition for $m^\epsilon$ in \eqref{eq:initial terminal}, we obtain that $m^\epsilon$ satisfies
	\begin{align*}
		\int_{\Rr} |x|^{\gamma} m^\epsilon(t,x) dx & \leq 2^{\gamma-1}  \int_{\Rr}|x|^{\gamma} m_0(x)dx + C_1' + C_2' \frac{(2T)^{\gamma/2}}{\pi}\Gamma(\tfrac{\zeta_1+1}{2}).
	\end{align*}
	By Assumption \ref{hyp: m_0 moment}, the right-hand side of the previous inequality is bounded independently of $\epsilon$, for $0<\epsilon<1$, as stated.
\end{proof}

Let $t \in [0,T]$. Define $\beta^\epsilon_t \in \Pp(\Rr^2)$ by
\begin{equation*}
  	\int_{\Rr^2} \psi(x,p)d\beta^\epsilon_t(x,p) = \int_{\Rr} \psi(x,\varpi^\epsilon(t) +u_x^\epsilon (t,x) ) m^\epsilon(t,x) \dx \quad \mbox{for all} ~ \psi \in C_{\bar{\zeta}}(\Rr^2),
 \end{equation*}
where $C_\zeta(\Rr^2)=\{\phi\in C(\Rr^2):  \lim_{|x|,|v| \to \infty} \frac{\phi(x,v)}{1+|x|^{\zeta_1} + |v|^{\zeta_2}} = 0 \}$ and $\bar{\zeta}=(\gamma_1,\gamma_2^\prime)$. Note that the well definiteness of the measure $\beta^\epsilon_t$ is ensured by Proposition \ref{prop: gamma1 moment mu}. Relying on the definition of $\beta^\epsilon_t$, we define $\mu^\epsilon_t \in \Pp(\Rr^2)$ by
 \begin{equation*}
 	\int_{\Rr^2} \psi(x,-L_v(x,v)) d \mu^\epsilon_t(x,v) = \int_{\Rr^2} \psi(x,p) d\beta^\epsilon_t(x,p) \quad \mbox{for all} ~ \psi \in C_{\bar{\zeta}}(\Rr^2).
 \end{equation*}
If Assumption \ref{hyp: H convex} holds, the relation $v=-H_p(x,p)$ if and only if $p=-L_v(x,v)$ (see Remark \ref{remark:H growth}), implies
 \begin{equation*}
  	\int_{\Rr^2} \psi(x,-H_p(x,p)) d\beta^\epsilon_t(x,p) 	=\int_{\Rr^2} \psi(x,v) d \mu^\epsilon_t(x,v).
 \end{equation*}
Finally, we define $\beta^\epsilon,\mu^\epsilon \in \mathcal{U}^{\bar{\zeta}} \cap \mathcal{R}^+(\Omega)$ by 
 \[
 	\int_{\Omega} f(t,x,v) d\beta^\epsilon(t,x,v) = \int_0^T \int_{\Rr^2} f(t,x,v) d\beta^\epsilon_t(x,v) \dt,
 \]  
 and
  \begin{equation}\label{def:mu epsilon}
   	\int_{\Omega} f(t,x,v) d\mu^\epsilon(t,x,v) = \int_0^T \int_{\Rr^2} f(t,x,v) d\mu^\epsilon_t(x,v)\dt,
  \end{equation}
 for all $f \in C_{\bar{\zeta}}(\Omega)$ (see Remark \ref{remark:Czeta dual}). Under Assumptions \ref{hyp: H convex}, \ref{hyp: H growth}, \ref{hyp:Q C infty}, \ref{hyp: H separability DS}, \ref{hyp: V-uT Lipschitz 2nd D bounded DS} and \ref{hyp: V-uT convex DS}, the non-negative and finite Radon measures $\mu^\epsilon$ defined by \eqref{def:mu epsilon} have a weak limit in $\mathcal{U}^{\bar{\zeta}}$ as $\epsilon \to 0$.

We show the existence of a weak limit of the Radon measures $\mu^\epsilon$ defined by \eqref{def:mu epsilon}.

\begin{proposition}\label{prop:weak limi mu}
Suppose Assumptions \ref{hyp: H convex}, \ref{hyp: H growth}, \ref{hyp:Q C infty}-
\ref{hyp: V-uT convex DS} hold. Then, there exists $\mu \in \mathcal{U}^{\bar{\zeta}} \cap \mathcal{R}^+(\Omega)$, where $\bar{\zeta}=(\gamma_1,\gamma_2^\prime)$, such that, up to a sub-sequence, the sequence of Radon measures $\mu^\epsilon$ defined by \eqref{def:mu epsilon} weakly converge to $\mu$; that is, for all $f \in C_{\bar{\zeta}}(\Omega)$
\begin{equation}\label{prop:weak limi mu-eq}
\begin{split}
\int_{\Omega} f(t,x,v) d\mu^\epsilon(t,x,v) \to 
\int_{\Omega} f(t,x,v) d\mu(t,x,v).
\end{split}
\end{equation}
\end{proposition}
\begin{proof}
By \eqref{eq:aux Hp bound} and Proposition \ref{prop: gamma1 moment mu}, we have
\begin{align*}
&\int_{0}^{T}	\int_{\Rr}\left(1+ |x|^{\gamma} + |H_p(x,\varpi^\epsilon(t)+u_x^\epsilon(t,x))|^{\gamma_2^\prime+1} \right) m^\epsilon(t,x) \dx \dt
\\
& \leq \int_{0}^{T}\int_{\Rr}\left(1+ |x|^{\gamma} +\tilde{C}^{\gamma_2^\prime+1}\right) m^\epsilon(t,x) \dx\dt \nonumber
	\\
	& \leq C (1+ \tilde{C}^{\gamma_2^\prime+1}).
\end{align*}
Using the previous inequality, an argument similar to that in Remark \ref{Remark:H(m0) not empty} shows that the probability measures $\mu^\epsilon$ defined by \eqref{def:mu epsilon} satisfy
\begin{align}\label{eq:mu epsilon zeta1 abs moment}
&\int_{\Omega} \left(1+|x|^{\gamma}+|v|^{\gamma_2^\prime+1}\right) d\mu^\epsilon(t,x,v) \nonumber
\\
&=\int_{0}^{T} \int_{\Rr^2}\left(1+|x|^{\gamma}+|v|^{\gamma_2^\prime+1}\right) d\mu^\epsilon_t(x,v)\dt \nonumber
\\
&\leq C (1+ \tilde{C}^{\gamma_2^\prime+1}),
\end{align}
where $C$ and $\tilde{C}$ are independent of $\epsilon$. Hence, $\mu^\epsilon\in \mathcal{U}^{\bar{\zeta}} \cap \mathcal{R}^+(\Omega)$, with $\bar{\zeta}=(\gamma_1,\gamma_2^\prime)$. Furthermore, \eqref{eq:mu epsilon zeta1 abs moment} implies that the measure $\nu^\epsilon=\left(1+|x|^{\gamma_1}+|v|^{\gamma_2^\prime}\right) \mu^\epsilon(t,x,v)$ belongs to $\mathcal{R}^+(\Omega)$ and 
\begin{equation*}
\int_{\Omega} \left(t^{\alpha_0}+|x|^{\alpha_0}+|v|^{\alpha_0}\right) d\nu^\epsilon(t,x,v)<C,
\end{equation*}
where $0<\alpha_0<\min\{(\gamma-\gamma_1),\frac{1}{\gamma}(\gamma_2^\prime+1)(\gamma-\gamma_1),\frac{\gamma}{\gamma_2^\prime+1},1 \}$.
 Therefore, as $\epsilon \to 0$, the sequence $\nu^\epsilon$ is tight (\cite{Werner}, Proposition 2.23). Hence, by Prohorov's Theorem (\cite{Werner}, Theorem 2.29), there exists $\nu \in \mathcal{R}^+(\Omega)$ such that, up to a sub-sequence, which we still denote by $\nu^\epsilon$, $\nu^\epsilon$ weakly converges to $\nu$; that is,
\begin{equation}\label{eq:mut weak conv}
	\int_{\Omega} \psi (t,x,v) d\nu^\epsilon(t,x,v) \to \int_{\Omega} \psi(t,x,v) d\nu(t,x,v) \quad \mbox{for all} \quad \psi \in C_b(\Omega).
\end{equation}
Now, taking $\mu=\frac{1}{1+|x|^{\gamma_1}+|v|^{\gamma_2^\prime}}\nu$, we notice that $\mu\in \mathcal{U}^{\bar{\zeta}} \cap \mathcal{R}^+(\Omega)$. Moreover, recalling the definition of $\nu^\epsilon$ from  \eqref{eq:mut weak conv}, we deduce \eqref{prop:weak limi mu-eq}.
\end{proof}

Next, we show that the weak limit provided by Proposition \ref{prop:weak limi mu} belongs to $\mathcal{H}(m_0)$. 
 \begin{proposition} \label{prop:mu epsilon in H}
Suppose Assumptions \ref{hyp: H convex}, \ref{hyp: H growth}, \ref{hyp:Q C infty}-
\ref{hyp: V-uT convex DS} hold. Let $\mu \in \mathcal{R}(\Omega)$ be such that, up to a sub-sequence, the Radon measures $\mu^\epsilon$ defined by \eqref{def:mu epsilon} weakly converge to $\mu$. Then, $\mu \in \mathcal{H}(m_0)$. 
\end{proposition}
\begin{proof}
The existence of $\mu$ is given by Proposition \ref{prop:weak limi mu}. 
By \eqref{eq:mu epsilon zeta1 abs moment}, we have that $\mu \in \mathcal{H}_1$. Let $(u,m,\varpi)$ be the solution of \eqref{eq:MFG system} and \eqref{eq:initial terminal} (\cite{gomes2018mean}, Theorem 1) . Let $\varphi \in C^1_c([0,T] \times \Rr)$. Because $m$ is a weak solution of the second equation in \eqref{eq:MFG system}, we have
\begin{align*}
	&\int_0^T \int_{\Rr} \big(\varphi_t(t,x) - H_p(x,\varpi + u_x)\varphi_x(t,x) \big)m(t,x) dx 
	\\
	&= \int_{\Rr} \varphi(T,x)m(T,x)\dx - \int_{\Rr} \varphi(0,x) m_0(x) \dx,
\end{align*}
and by \eqref{def:mu epsilon}
\begin{align*}
	\int_{\Omega} \varphi_t (t, x) + v \varphi_x(t, x) d\mu^\epsilon (t,x,v) & = \int_0^T \int_{\Rr^2} \varphi_t (t, x) + v \varphi_x(t, x) ~d\mu^\epsilon_t (x,v)dt 
	\\
	&= \int_0^T \int_{\Rr^2} \varphi_t (t, x) - H_p(x,p) \varphi_x(t, x) ~d\beta^\epsilon_t (x,p)dt 
	\\
	&=\int_0^T \int_{\Rr} \varphi_t (t, x) - H_p(x,\varpi^\epsilon + u_x^\epsilon) \varphi_x(t, x) ~m^\epsilon(t,x) dx dt.
\end{align*}
Now, taking into account \eqref{eq:mu epsilon zeta1 abs moment} and arguing as in Remark \ref{Remark:H(m0) not empty}, we deduce that the previous two identities also hold for any $\varphi \in \Lambda([0,T] \times \Rr)$. Hence, $\mu \in \mathcal{H}_2(m_0,\nu)$ for $\nu = m(T,\cdot) \in \Pp(\Rr)$. 
Finally, the third equation in \eqref{eq:regularized MFG} gives
\begin{equation*}
	\int_{\Omega} \eta(t) (v-Q(t)) d\mu^\epsilon (t,x,v) = \int_0^T \int_{\Rr}  \eta(t) (-H_p(x,\varpi^\epsilon+u_x^\epsilon)-Q(t)) m^\epsilon (t,x)dx dt=0,
\end{equation*}
for all $\eta \in C([0,T])$, which implies that $\mu \in \mathcal{H}_3$. Therefore, $\mu \in \mathcal{H}(m_0)$ as stated.
\end{proof}

Next, we prove the following technical lemma.

	\begin{lemma}\label{int-rep} Let $x_n>y_n$, $y_n\to+\infty$ and $\lim\limits_{n\to\infty}\frac{x_n}{y_n}=1$. Suppose that $\phi\in L^1(\Rr)$ and  $\int_{\Rr}|x|^\sigma\phi \dx <\infty$ for some $\sigma>0$ . Then,
		\begin{equation}\label{lll-eq}
		\lim\limits_{n\to\infty}\frac{1}{2x_n}\int_{-y_n}^{x_n}\int_{-y_n+y}^{x_n+y}\phi(x)\dx\dy=\int_{\Rr}\phi(x)\dx.
		\end{equation}
	\end{lemma}
	\begin{proof} After exchanging the order of the integrals on the left-hand side in \eqref{lll-eq}, we have
		\begin{equation*}
		\begin{split}
		&\int_{-y_n}^{x_n}\int_{-y_n+y}^{x_n+y}\phi(x)\dx\dy= \int_{-2y_n}^{0}\phi(x)\int_{-y_n}^{x+y_n}\dy\dx \\& +
		\int_{0}^{x_n-y_n}\phi(x)\int_{-y_n}^{x_n}\dy\dx+	\int_{0}^{x_n-y_n}\phi(x)\int_{x+y_n}^{x_n}\dy\dx+
		\int_{x_n-y_n}^{2x_n}\phi(x)\int_{-x_n+x}^{x_n}\dy\dx\\&=2x_n \int_{-2y_n}^{2x_n}\phi(x)\dx+2(y_n-x_n)\int_{-2y_n}^{0}\phi(x)\dx+\int_{-2y_n}^{0}\phi(x)x\dx-\int_{0}^{2x_n}\phi(x)x\dx .
		\end{split}
		\end{equation*}
		Dividing  the proceeding equation by $2x_n$ and letting $n\to\infty$, we deduce \eqref{lll-eq}.
\end{proof}

Now, relying on the previous results, we complete the second part of the proof of Theorem \ref{thm: Representation}. This is the content of the following Lemma.
\begin{lemma}
\label{lem: upper bound}
Suppose Assumptions \ref{hyp: H convex}-
\ref{hyp: V-uT convex DS} hold. Let $(u,m,\varpi)$ solve \eqref{eq:MFG system} and \eqref{eq:initial terminal}. 
	Then
	\begin{equation*}
	\int_{\Rr} \left( u(0, x) - u_T(x) \right) dm_0(x) - \int_0^T \varpi(t) Q(t) dt \geq 
	\inf_{\substack{\mu \in \mathcal{H}(m_0)}} \int_0^T \int_{\Rr^2} L(x,v) + v u_T'(x)d\mu(t,x,v).
	\end{equation*}
\end{lemma}

\begin{proof}
	By Assumption \ref{hyp: H convex} and \eqref{eq:Legendre transform}, the following identity holds
	\begin{equation}\label{eq:aux L identity}
	L(x, v) = H_p(x, - L_v(x, v)) (-L_v(x, v)) - H(x, -L_v(x, v)).
	\end{equation}
	Let $t\in [0,T]$. By Remark \ref{remark:H growth} and  \eqref{eq:aux Hp bound}, we have
	\[
	\int_{\Rr} L(x,-H_p(x,\varpi^\epsilon+u_x^\epsilon)) m^\epsilon(t,x)dx \leq \int_{\Rr} \left( C_2|x|^{\gamma_1} + C\right)m^\epsilon(t,x) dx,
	\]
	where $C$ is independent of $x$ and $\epsilon$. From the previous inequality, Assumption \ref{hyp: m_0 moment} and Proposition \ref{prop: gamma1 moment mu}, and an argument similar to that in Remark \ref{Remark:H(m0) not empty}, we get that the integral $\int_{\Rr^2}L(x,v) d\mu^\epsilon_t(x,v)$ exists and is finite. Hence, we integrate both sides of \eqref{eq:aux L identity} w.r.t. $\mu^\epsilon_t$, and we use the definition of $\beta_t^\epsilon$ to obtain
	\begin{align}\label{eq:Aux -2 Upper bound Thm1.2}
	\int_{\Rr^2} L(x, v) d \mu^\epsilon_t(x,v) &= \int_{\Rr^2} H_p(x, - L_v(x, v)) (-L_v(x, v)) - H(x, -L_v(x, v))d \mu^\epsilon_t(x,v) \nonumber
	\\
	& =  \int_{\Rr^2} H_p(x, p)p - H(x, p)d \beta^\epsilon_t(x,p) \nonumber
	\\
	& = \int_{\Rr} \big(H_p(x, \varpi^\epsilon + u^\epsilon_x) (\varpi^\epsilon + u^\epsilon_x) - H(x, \varpi^\epsilon + u^\epsilon_x) \big) m^\epsilon d x.
	\end{align}
	
	Let $a\in[-\frac{1}{2},0]$ be such that $|m^\varepsilon(t,a)|<\infty$. 
	By Proposition \ref{prop: gamma1 moment mu} and Assumption \ref{hyp: m_0 moment}, we have that $\int_{\Rr}|x|m^\epsilon dx <\infty$. Rewriting the first momentum of $m^\varepsilon$
	\begin{equation*}
	\int_{\Rr}|x|m^\epsilon \dx=\int_{-a}^{0}|x|m^\epsilon \dx+\sum_{n=0}^{\infty}\int_{n}^{n+1}xm^\epsilon \dx+\int_{-n-a-1}^{-n-a}|x|m^\epsilon dx<\infty,
	\end{equation*}
	we deduce that there exists $N_0$ such that for all $N\geq N_0$ 
	\begin{equation*}
	\int_{N}^{N+1}xm^\epsilon dx+\int_{-N-a-1}^{-N-a}|x|m^\epsilon \dx=\int_{N}^{N+1}xm^\epsilon dx+\int_{-N-1}^{-N}|x-a|m^\epsilon(t,x-a) \dx\leq \frac{C}{N}.
	\end{equation*}
	The previous estimates with Chebyshev's inequality imply
	\begin{equation*}
	\begin{split}
	\bigg|\{x\in[N,N+1]:\quad xm^\epsilon>\frac{1}{\sqrt{N}}\} \bigg|\leq \sqrt{N}\int_{N}^{N+1}xm^\epsilon dx\leq \frac{C}{\sqrt{N}},\\
	\bigg|\{x\in[-N-1,-N]:\quad |x-a|m^\epsilon(t,x-a)>\frac{1}{\sqrt{N}}\} \bigg|\leq \sqrt{N}\int_{-N-a-1}^{-N-a}|x|m^\epsilon dx\leq \frac{C}{\sqrt{N}}.
	\end{split}
	\end{equation*}
	Because $a\in[-\frac{1}{2},0]$,  there exists a sequence $\{x_n\}$ such that
	\begin{equation}\label{sequences}
	\begin{split}
	&x_{n}\geq 0 ,\quad \lim\limits_{n\to\infty} x_n=+\infty,\\
	&\lim\limits_{n\to\infty} x_nm^\epsilon(t,2x_n)= \lim\limits_{n\to\infty} x_nm^\epsilon(t,-2x_n-2a)=0.\\
	\end{split}
	\end{equation}
	Let $y_n=x_n+a$.
	
By Assumption \ref{hyp: m_0 moment} follows that there exists $\sigma>0$ such that  $\int_{\Rr}|x|^{\gamma_1+\sigma}m_0\dx<\infty$. Then, relying on Proposition \ref{prop: gamma1 moment mu} and 
	using Lemma \ref{int-rep}, we  rewrite \eqref{eq:Aux -2 Upper bound Thm1.2}
	\begin{equation}\label{eq:Aux -3 Upper bound Thm1.2}
	\begin{split}
	&\int_{\Rr^2} L(x, v) d \mu^\epsilon_t(x,v) =\lim\limits_{n\to\infty}\frac{1}{2x_n}\int_{-y_n}^{x_n}\int_{-y_n+y}^{x_n+y} \big(H_p(x, \varpi^\epsilon + u^\epsilon_x) (\varpi^\epsilon + u^\epsilon_x)\big. \\&\big.- H(x, \varpi^\epsilon + u^\epsilon_x) \big) m^\epsilon \dx\dy=
	\lim\limits_{n\to\infty}\frac{1}{2x_n}\int_{-y_n}^{x_n}\int_{-y_n+y}^{x_n+y}   H_p(x, \varpi^\epsilon+ u^\epsilon_x) m^\epsilon  \big(u^\epsilon -  u_T\big)_x\\&+H_p(x, \varpi^\epsilon + u^\epsilon_x) \big(\varpi^\epsilon +  u_T'\big) m^\epsilon-H(x, \varpi^\epsilon + u^\epsilon_x) m^\epsilon\dx\dy.
	\end{split}
	\end{equation}
	Because $H$ is separable, $u^\epsilon(t,\cdot),u_T\in \mbox{Lip}(\Rr)$, from Proposition \ref{prop: gamma1 moment mu}, we deduce 
	\begin{equation*}
	\lim\limits_{n\to\infty}\frac{1}{2x_n}\left| \int_{-y_n}^{x_n}  H_p(x, \varpi^\epsilon+ u^\epsilon_x) m^\epsilon  \big(u^\epsilon -  u_T\big)\bigg|_{-y_n+y}^{x_n+y}\dy\right|\leq \lim\limits_{n\to\infty}\frac{C}{x_n}\int_{\Rr}|x|m^\epsilon\dx =0.
	\end{equation*}
	Integrating by parts the first term on the right-hand side in \eqref{eq:Aux -3 Upper bound Thm1.2}, using the preceding equality, and the definition of $\beta^\epsilon_t$, \eqref{eq:Aux -3 Upper bound Thm1.2} becomes 
	\begin{align}\label{eq:Aux 0 Upper bound Thm1.2} 
	& \int_{\Rr^2} L(x, v) d \mu^\epsilon_t(x,v) \nonumber
	\\
	& = \lim\limits_{n\to\infty}\frac{1}{2x_n}\int_{-y_n}^{x_n}\int_{-y_n+y}^{x_n+y} - \big(H_p(x, \varpi^\epsilon + u^\epsilon_x) m^\epsilon\big)_x \big(u^\epsilon -u_T\big) -H(x, \varpi^\epsilon + u^\epsilon_x) m^\epsilon \dx\dy \nonumber
	\\
	&  + \lim\limits_{n\to\infty}\frac{1}{2x_n}\int_{-y_n}^{x_n}\int_{-y_n+y}^{x_n+y}\int_\Rr  H_p(x,p) \big(\varpi^\epsilon +  u_T'\big)  d\beta^\epsilon_t(x,p) \dy.
	\end{align} 
	
	Next, we prove well definiteness of several integrals. Note that  Assumption \ref{hyp: m_0 moment}  and the definition of $\mu^\epsilon_t$, yield  
	\begin{equation*}
	\begin{split}
	\left| \int_{\Rr^2}  v|x|^\sigma\big(\varpi^\epsilon +  u_T'\big)  d\mu^\epsilon_t(x,v)\right| &=\left| \int_{\Rr^2}  H_p(x,p) |x|^\sigma\big(\varpi^\epsilon +  u_T'\big)  d\beta^\epsilon_t(x,p)\right|\\&= \left| \int_{\Rr}  H_p(x,\varpi^\epsilon +  u_x^\epsilon) |x|^\sigma\big(\varpi^\epsilon +  u_T'\big) m^\epsilon  \dx\right|\leq C,
	\end{split}
	\end{equation*}
for $\gamma_1<\gamma_1+\sigma<\gamma$.
Relying on	the preceding estimate and considering Lemma \ref{int-rep}, we obtain
	\begin{equation}\label{key-H_p}
	-\lim\limits_{n\to\infty}\frac{1}{2x_n}\int_{-y_n}^{x_n}\int_{-y_n+y}^{x_n+y}\int_\Rr  v\big(\varpi^\epsilon +  u_T'\big)  d\mu^\epsilon_t(x,v)=-\int_{\Rr^2} v\big(\varpi^\epsilon +  u_T'\big)  d\mu^\epsilon_t(x,v).
	\end{equation}
By the second-order energy estimate in  \eqref{second-eng-est} and using Young's inequality, we have
\begin{equation*}
\begin{split}
		\int_0^T \int_{\Rr} |u_{xx}^\epsilon| |x|^{\frac{\gamma}{2}} m^\epsilon \dx \dt\leq \int_0^T \int_{\Rr} (u_{xx}^\epsilon)^2  m^\epsilon +|x|^{\gamma} m^\epsilon\dx \dt \leq C.
\end{split}
\end{equation*}
Hence, Lemma  \ref{int-rep} implies that 
\begin{equation*}
	\begin{split}
\int_0^T \int_{\Rr} |u_{xx}^\epsilon| m^\epsilon \dx \dt =\lim\limits_{n\to\infty}\frac{1}{2x_n}\int_0^T \int_{-y_n}^{x_n}\int_{-y_n+y}^{x_n+y}   |u_{xx}^\epsilon| m^\epsilon\dx \dt  \leq C.
	\end{split}
\end{equation*}
	Using the previous estimate and taking into account that $u^\epsilon\in \mbox{Lip}(\Rr)$ for all $t\in[0,T]$, we get
	\begin{equation}\label{sequences2}
	\begin{split}
	& \lim\limits_{n\to\infty}\frac{1}{2x_n}\left| \int_{0}^{T}\int_{-y_n}^{x_n}\int_{-y_n+y}^{x_n+y}m_x^\epsilon u^\epsilon_x \dx\dy \dt\right| \leq   \lim\limits_{n\to\infty}\frac{1}{2x_n}\left|\int_{0}^{T} \int_{-y_n}^{x_n}m^\epsilon(t,x_n+y) u^\epsilon_x(t,x_n+y)\right. \\&\left. -m^\epsilon(t,-y_n+y) u^\epsilon_x(t,-y_n+y) \dy\dt\bigg|\right. + \lim\limits_{n\to\infty}\frac{1}{2x_n}\int_{0}^{T}\int_{-y_n}^{x_n}\int_{-y_n+y}^{x_n+y} |u_{xx}^\epsilon|m^\epsilon  \dx\dt \\&\leq
	\lim\limits_{n\to\infty}\frac{1}{x_n}\int_{0}^{T}\int_{\Rr} m^\epsilon |u_{x}^\epsilon|\dy \dt+\int_{0}^{T}\int_{\Rr} |u_{xx}^\epsilon| m^\epsilon  \dy\dt\leq C.
	\end{split}
	\end{equation}
	Because $|m^\varepsilon(t,a)|<\infty$ from  \eqref{sequences},  we have
	\begin{equation}\label{key-0}
	\begin{split}
	&\lim\limits_{n\to\infty}\frac{1}{2x_n}\left|\int_{-y_n}^{x_n}  m^\epsilon_x u^\epsilon\bigg|_{-y_n+y}^{x_n+y}\dy \right|\\&\leq\lim\limits_{n\to\infty}\frac{1}{2x_n}\left|\int_{-y_n}^{x_n}  m^\epsilon_x(t,x_n+y) u^\epsilon(t,x_n+y)\dy-\int_{-y_n}^{x_n}  m^\epsilon_x(t,-y_n+y) u^\epsilon(t,-y_n+y)\dy  \right|\\&=\lim\limits_{n\to\infty}\frac{1}{2x_n}\left|\int_{x_n-y_n}^{2x_n}  m^\epsilon_x(t,y) u^\epsilon(t,y)\dy-\int_{-2y_n}^{x_n-y_n}  m^\epsilon_x(t,y) u^\epsilon(t,y)\dy  \right|\\&\leq\lim\limits_{n\to\infty}\frac{1}{2x_n} \left( m^\epsilon(t,2x_n) |u^\epsilon(t,2x_n)|+m^\epsilon(t,-2y_n)|u^\epsilon(t,-2y_n)|+2m^\epsilon(t,a)|u^\epsilon(t,a)|\right. \\& +2\int_{\Rr}m^\epsilon|u_x^\epsilon|\dx)=0.
	\end{split}
	\end{equation}
	Furthermore,  \eqref{sequences2} and \eqref{key-0},  yield
	\begin{equation}\label{sequences4}
	\begin{split}
	\left| \lim\limits_{n\to\infty}\frac{1}{2x_n}\int_{0}^{T}\int_{-y_n}^{x_n}\int_{-y_n+y}^{x_n+y}m^\epsilon_{xx} u^\epsilon \dx\dy\dt\right|  &\leq\left|\lim\limits_{n\to\infty}\frac{1}{2x_n}\int_{0}^{T}\int_{-y_n}^{x_n}  m^\epsilon_x u^\epsilon \bigg|_{-y_n+y}^{x_n+y}\dy\dt \right| \\&+\left| \lim\limits_{n\to\infty}\frac{1}{2x_n}\int_{0}^{T}\int_{-y_n}^{x_n}\int_{-y_n+y}^{x_n+y}m^\epsilon_xu^\epsilon_x \dx\dy\dt\right| \leq C.
	\end{split}
	\end{equation}
Note that \eqref{sequences2}, \eqref{key-0} and \eqref{sequences4} also hold for $u_T$.

	Because  $\varpi^\epsilon\in C[0,T]$, then $H(x, \varpi^\epsilon + u^\epsilon_x) \in L^\infty_{loc}([0,T]\times\Rr)$, which with the regularity of heat equation  implies that $u_x\in C([0,T]\times\Rr)$ and $u_{xx}\in L_{loc}^p([0,T]\times\Rr)$ for every $p\in[1,\infty)$.
	Therefore,  the second and the first equation in \eqref{eq:regularized MFG} imply
	\begin{gather*}   
	-\big(H_p(x, \varpi^\epsilon + u^\epsilon_x) m^\epsilon\big)_x \big(u^\epsilon -u_T\big) = \left(\epsilon m^\epsilon_{xx} -m^\epsilon_t\right)\big(u^\epsilon -u_T\big),
	\\
	-H(x, \varpi^\epsilon + u^\epsilon_x) m^\epsilon=-\left(\epsilon u^\epsilon_{xx}+u^\epsilon_t\right)m^\epsilon.
	\end{gather*}  
Relying on \eqref{sequences4} and using the preceding identities and  the identities in \eqref{key-H_p} after integrating on $[0,T]$ the equation in \eqref{eq:Aux 0 Upper bound Thm1.2}, we obtain
	\begin{equation}\label{eq:aux 22 Upper bound Thm1.2}
	\begin{split}
	& \int_{\Omega} L(x, v) d \mu^\epsilon_t(x,v) \dt
	= -\int_{0}^{T}\int_{\Rr^2} v\big(\varpi^\epsilon +  u_T'\big)  d\mu^\epsilon_t(x,v)\dt\\&+\lim\limits_{n\to\infty}\frac{1}{2x_n}\int_{0}^{T}\int_{-y_n}^{x_n}\int_{-y_n+y}^{x_n+y}- m^\epsilon_t \big(u^\epsilon-u_T\big) + \epsilon m^\epsilon_{xx} \big(u^\epsilon-u_T\big)  -u^\epsilon_t m^\epsilon - \epsilon u^\epsilon_{xx} m^\epsilon\dx \dy\dt
	\\&= \lim\limits_{n\to\infty}\frac{1}{2x_n}\int_{0}^{T}\int_{-y_n}^{x_n}\int_{-y_n+y}^{x_n+y}- \big(m^\epsilon \big(u^\epsilon-u_T\big)\big)_t + \epsilon m^\epsilon_{xx} \big(u^\epsilon-u_T\big)- \epsilon u^\epsilon_{xx} m^\epsilon\dx \dy\dt\\&
	-\int_{\Omega} v\big(\varpi^\epsilon +  u_T'\big)  d\mu^\epsilon_t(x,v)\dt.
	\end{split}
	\end{equation}
	Taking into account \eqref{key-0} and  integrating by parts, we have
	\begin{align}\label{key-1}
	&\lim\limits_{n\to\infty}\frac{1}{2x_n}\left| \int_{0}^{T}\int_{-y_n}^{x_n}\int_{-y_n+y}^{x_n+y}  m^\epsilon_{xx} u^\epsilon-  u^\epsilon_{xx} m^\epsilon\dx \dy\dt\right|  \nonumber
	\\
	&=\lim\limits_{n\to\infty}\frac{1}{2x_n}\left| \int_{0}^{T}\int_{-y_n}^{x_n} m^\epsilon_{x} u^\epsilon-  u^\epsilon_{x} m^\epsilon\bigg|_{-y_n+y}^{x_n+y}\dy\dt\right|  \nonumber
	\\
	&\leq \lim\limits_{n\to\infty}\frac{1}{2x_n}\left( \left|\int_{0}^{T} \int_{-y_n}^{x_n} m^\epsilon_{x} u^\epsilon\bigg|_{-y_n+y}^{x_n+y}\dy\dt\right| + C \int_{0}^{T}\int_{\Rr} m^\epsilon\dx\dt\right) \nonumber
	\\
	&=0.
	\end{align}
	Thus, recalling the definition of $\mu^\epsilon$, \eqref{def:mu epsilon}, and by using   \eqref{sequences4}, \eqref{key-1} in \eqref{eq:aux 22 Upper bound Thm1.2}, we get
	\begin{align*}
	\int_{\Omega} L(x, v) d \mu^\epsilon(t,x,v) 
	&= \lim\limits_{n\to\infty}\frac{1}{2x_n}\int_{0}^{T}\int_{-y_n}^{x_n}\int_{-y_n+y}^{x_n+y} -\left( m^\epsilon\left( u^\epsilon-u_T\right) \right)_t -\epsilon m^\epsilon_{xx} u_T \dx\dy\dt 
	\\
	&  -\int_{\Omega} v\big(\varpi^\epsilon +  u_T'\big)  d\mu^\epsilon(t,x,v).
	\end{align*} 
	After rearranging the terms in  the previous equation, we obtain
	\begin{align}\label{eq:aux limit identity}
	& \int_{\Omega} L(x, v)  + v u_T' d \mu^\epsilon(t,x,v) \nonumber
	=  \lim\limits_{n\to\infty}\frac{1}{2x_n}\int_{0}^{T}\int_{-y_n}^{x_n}\int_{-y_n+y}^{x_n+y} \big(u^\epsilon(0,x) - u_T(x)\big)  m_0\dx\dy\dt \\& - \epsilon \lim\limits_{n\to\infty}\frac{1}{2x_n}\int_{0}^T\int_{-y_n}^{x_n}\int_{-y_n+y}^{x_n+y} m^\epsilon_{xx} u_T \dx\dy \dt   -\int_{\Omega} v\varpi^\epsilon  d \mu^\epsilon(t,x,v).
	\end{align}
	
	Now, we pass to the limit in \eqref{eq:aux limit identity} as follows. By Assumptions \ref{hyp: H convex}, \ref{hyp:Q C infty}, \ref{hyp: H separability DS} and \ref{hyp: V-uT Lipschitz 2nd D bounded DS}, Theorem 1 in \cite{gomes2018mean} guarantees the existence of a sequence such that $u^\epsilon \to u$ and $\varpi^\epsilon \to \varpi$ uniformly, where, for $\varpi$, $u$ solves the first equation in \eqref{eq:MFG system} in the viscosity sense. Furthermore, by Proposition \ref{prop:price Lipschitz}, $\varpi \in W^{1,\infty}([0,T])$.   Remark \ref{hyp: H convex} implies that $L(x, v) + v u_T'(x)$  belongs to $C_{\bar{\zeta}}(\Omega)$, therefore  extracting a further sub-sequence out of the previous sequence, Proposition \ref{prop:weak limi mu} gives the existence of a weak limit $\mu\in\mathcal{U}^{\bar{\zeta}}\cap\mathcal{R}(\Omega)$ for $\mu^\epsilon$ and \eqref{prop:weak limi mu-eq} holds for $L(x, v) + v u_T'(x)$. Using these, by letting $\epsilon \to 0$ in \eqref{eq:aux limit identity}, we obtain
	\begin{align}\label{eq:aux mu identity}
	\int_{0}^T\int_{\Rr^2}  L(x, v) + v u_T'(x) d \mu(t,x,v) 
	& =  \lim\limits_{n\to\infty}\frac{1}{2x_n}\int_{0}^T\int_{-y_n}^{x_n}\int_{-y_n+y}^{x_n+y}( u(0, x) - u_T(x) ) m_0\dx\dy  \notag\\&- \int_{\Omega} v \varpi(t) d \mu(t,x,v).
	\end{align}
	Furthermore, by Proposition \ref{prop:price Lipschitz}, $\varpi \in W^{1,\infty}([0,T])$, and by Proposition \ref{prop:mu epsilon in H} $\mu \in \mathcal{H}(m_0)$. In particular, $\mu \in \mathcal{H}_3$. Therefore, 
	\[
	\int_{0}^T\int_{\Rr^2}  v \varpi(t) d \mu(t,x,v)= \int_0^T Q(t) \varpi(t) \dt.
	\]
	By Assumption \ref{hyp: m_0 moment} and Lemma \ref{int-rep}, we deduce that 
	\begin{equation*}
	\lim\limits_{n\to\infty}\frac{1}{2x_n}\int_{0}^T\int_{-y_n}^{x_n}\int_{-y_n+y}^{x_n+y}( u(0, x) - u_T(x) ) m_0\dx\dy =\int_{0}^{T}\int_{\Rr} u(0, x) - u_T(x) d m_0(x). 
	\end{equation*}
	Therefore,   from \eqref{eq:aux mu identity}, we obtain
	\begin{align*}
	\int_{\Omega} L(x, v) + v u_T'(x) d \mu(t,x,v) 
	& = \int_{\Rr} u(0, x) - u_T(x) d m_0(x)  - \int_0^T Q(t) \varpi(t) dt,
	\end{align*}
which completes the proof.
\end{proof}

%

\bibliographystyle{plain}
\bibliography{mfg.bib}

\end{document}